\documentclass[letterpaper, 10 pt, conference]{IEEEconf}
\usepackage{times}

\IEEEoverridecommandlockouts                              

\usepackage{subcaption}
\usepackage{amssymb}
\usepackage{multicol}
\usepackage[bookmarks=true]{hyperref}
\usepackage{amsmath,amsfonts}
\usepackage{array}
\usepackage{textcomp}
\usepackage{stfloats}
\usepackage{url}
\usepackage{verbatim}
\usepackage{graphicx}
\usepackage{mathtools}
\usepackage{bm}
\usepackage{xcolor}
\let\labelindent\relax
\usepackage{enumitem}
\usepackage{algorithm}
\usepackage{algpseudocode}
\usepackage{microtype}
\usepackage{todonotes}
\usepackage{cite}
\allowdisplaybreaks

\algnewcommand{\IfThenElse}[3]{
  \State \algorithmicif\ #1\ \algorithmicthen\ #2\ \algorithmicelse\ #3}
\algnewcommand{\IIf}[1]{\State\algorithmicif\ #1\ \algorithmicthen}
\algnewcommand{\EndIIf}{\unskip\ \algorithmicend\ \algorithmicif}

\algrenewcommand\algorithmicrequire{\textbf{Input:}}
\algrenewcommand\algorithmicensure{\textbf{Output:}}

\newtheorem{theorem}{Theorem}

\newtheorem{remark}{Remark}
\newtheorem{proposition}{Proposition}
\newtheorem{corollary}{Corollary}

\begin{document}

\title{\LARGE \bf Line-Search Filter Differential Dynamic Programming for Optimal Control with Nonlinear Equality Constraints}


\author{Ming Xu$^{1}$, Stephen Gould$^{2}$ and Iman Shames$^{3}$
\thanks{$^{1}$Ming Xu is with the School of Computer and Communication Sciences, EPFL (e-mail: mingda.xu@epfl.ch). }
\thanks{$^{2}$Stephen Gould is with the School of Computing at the Australian National University (e-mail: stephen.gould@anu.edu.au). }
\thanks{$^{3}$Iman Shames is with the Department of Electrical and Electronic Engineering, the University of Melbourne (email: iman.shames@unimelb.edu.au).}
\thanks{This work was supported by the Australian Research Council under grant DP250101763 and the United States Air Force Office of Scientific Research
under Grant No. FA2386-24-1-4014.}}







\maketitle

\begin{abstract}
We present FilterDDP, a differential dynamic programming algorithm for solving discrete-time, optimal control problems (OCPs) with nonlinear equality constraints. Unlike prior methods based on merit functions or the augmented Lagrangian class of algorithms, FilterDDP uses a step filter in conjunction with a line search to handle equality constraints. We identify two important design choices for the step filter criteria which lead to robust numerical performance: 1) we use the Lagrangian instead of the cost in the step acceptance criterion and, 2) in the backward pass, we perturb the value function Hessian. Both choices are rigorously justified, for 2) in particular by a formal proof of local quadratic convergence. In addition to providing a primal-dual interior point extension for handling OCPs with both equality and inequality constraints, we validate FilterDDP on three contact implicit trajectory optimisation problems which arise in robotics.
\end{abstract}


\section{Introduction}

Discrete-time optimal control problems (OCPs) are used in robotics for motion planning tasks such as minimum-time quadrotor flight \cite{foehnquadrotorscience}, autonomous driving \cite{gohdrivingmpc}, locomotion for legged robots \cite{wensingleggedsurvey, farshidianleggedmpc, grandiaperceptivelocomotion}, obstacle avoidance \cite{zengcollisionbarrier, zhangcollisionopt, marcucciconvexobstacles} and manipulation~\cite{xuenonprehensile, mouranonprehensile, Yang-RSS-24, hogannonprehensileijrr}. The differential dynamic programming (DDP) class of algorithms \cite{mayneddpbook} are of interest to the robotics community due to their efficiency, as well as other benefits for control, such as dynamic feasibility of iterates and a feedback policy provided as a byproduct of the algorithm \cite{grandiafeedbackddp, mastallinullspace}.

To date, extending DDP for nonlinear equality constraints is relatively underexplored. Equality constraints are required for OCPs involving inverse dynamics \cite{mastallinullspace, katayamainverse}, challenging contact implicit planning problems \cite{manchestercontactto, posalcpcontact, mouranonprehensile, howellopt, lecleachcimpc, Yang-RSS-24} and minimum time waypoint flight \cite{foehnquadrotorscience}. Most existing algorithms for handling equality constraints adopt an adaptive penalty approach based on the \emph{augmented Lagrangian} (AL) \cite{eqddpkazdadi, proxddp, altro}. However, AL methods yield limited formal convergence guarantees \cite[Ch. 17]{nwopt}, motivating a DDP algorithm based on a framework with stronger convergence properties, i.e., a line-search filter approach \cite{wachterglobal, wachterlocal}, which is adopted in the popular nonlinear programming solver IPOPT \cite{ipopt}.

As a result, we propose FilterDDP, a DDP line-search filter algorithm which complements existing works based on merit functions for handling equality constraints \cite{mastallinullspace, prabhu2024differentialdynamicprogrammingstagewise}. We identify two important design choices when designing the filter algorithm which are essential for robust numerical performance: 1) we replace the cost with the Lagrangian as one of the two filter criterion and, 2) for the stopping criteria and backward pass Hessians, we replace the value function gradient with an estimated dual variable of the dynamics constraints. Both 1) and 2) are justified rigorously, with 2) via a proof of local quadratic convergence. We also validate 1) and 2) with an ablation study in the simulation results.

Another contribution of this paper relates to the aforementioned theoretical analysis of constrained DDP algorithms. 
First, we mathematically show that one of FilterDDP's sufficient decrease conditions is analogous to the Armijo condition \cite[pg. 33]{nwopt} in unconstrained optimisation.  
Second, we provide a proof of local quadratic convergence of the iterates within a neighbourhood of a solution. 
This proof generalises the existing proof of local quadratic convergence of unconstrained DDP for scalar states and controls \cite{liao1992advantages} to the the constrained, vector-valued setting. 
Our final contribution is a primal-dual interior point extension to FilterDDP for handling inequality constraints as well as equality constraints.

An implementation of FilterDDP is provided in the Julia programming language\footnote{https://github.com/mingu6/FilterDDP.jl}. Our numerical experiments evaluate FilterDDP on 300 OCPs derived from three challenging classes of motion planning problems with nonlinear equality constraints: 1) a contact-implicit cart-pole swing-up task with Coulomb friction \cite{howellopt}, 2) an acrobot swing-up task with joint limits enforced through a variational, contact-implicit formulation \cite{howellopt} and, 3) a contact-implicit non-prehensile pushing task \cite{mouranonprehensile}. Our experiments indicate that FilterDDP offers significant robustness gains and lower iteration count compared to the existing state-of-the-art equality-constrained DDP algorithms \cite{proxddp, mastallinullspace}, while being faster (and more amenable to embedded systems) compared to IPOPT \cite{ipopt}.

\subsection{Notation used in this paper}

Let $[N]=\{1,2,\dots,N\}$. A function $p(x):\mathcal{X}\rightarrow
\mathcal{P}$ for some sets $\mathcal{X}$ and $\mathcal{P}$ is defined as being
$O(\|x\|^d)$ for positive integer $d$ if there exists a scalar $C$ such that
$\|p(x)\| \leq C \|x\|^d$ for all $x\in\mathcal{X}$. A vector of ones of
appropriate size is denoted by $e$.  Denote the collection of indexed vectors
$\{y_k\}_{k=t}^{N}$ by $y_{t:N}$ and we use the convention $(x, u)
= (x^\top, u^\top)^\top$. Denote the element-wise product by $\odot$. Finally, we use the following convention for derivatives: If $f$ is a scalar valued function $f : \mathbb{R}^n \times \mathbb{R}^m \rightarrow \mathbb{R}$, then $\nabla_u f(x, u) \in \mathbb{R}^{1\times m}$ and $\nabla_{xu}^2 f(x, u) \in \mathbb{R}^{n\times m}$. If $g:\mathbb{R}^n\times \mathbb{R}^m \rightarrow \mathbb{R}^c$ is a vector valued function, then $\nabla_x g(x, u) \in \mathbb{R}^{c\times n}$ and $\nabla_{xu}^2 g(x, u) \in\mathbb{R}^{c \times n \times m}$. 

\section{Discrete-time Optimal Control}\label{sec:oc}

In this section, we provide some background on equality constrained discrete-time optimal control problems (OCPs).

\subsection{Problem Formulation}

We consider OCPs with a finite-horizon length $N$ with equality constraints, which can be expressed in the form
\begin{equation}\label{eq:oc}
\begin{array}{rl}
    \underset{\mathbf{x}, \mathbf{u}}{\text{minimise}}  & J(x_{1:N}, u_{1:N}) \coloneqq \sum_{t=1}^{N} \ell(x_t, u_t) \\
    \text{subject to} & x_1 = \hat{x}_1, \\
    & x_{t+1} = f(x_t, u_t) \quad \text{for } t \in [N-1], \\
    & c(x_t, u_t) = 0  \quad \text{for } t \in [N],
\end{array}
\end{equation} where $x_{1:N}$ and $u_{1:N}$ are a trajectory of states and control inputs,
respectively, and we assume $x_t\in\mathbb{R}^{n_x}$ and $u_t\in\mathbb{R}^{n_u}$ for all $t$. The costs are denoted by
$\ell:\mathbb{R}^{n_x} \times\mathbb{R}^{n_u} \rightarrow \mathbb{R}$ and the constraints are denoted by $c:\mathbb{R}^{n_x}\times \mathbb{R}^{n_u} \rightarrow \mathbb{R}^{n_c}$
where $n_c \leq n_u$. The mapping $f:\mathbb{R}^{n_x} \times\mathbb{R}^{n_u} \rightarrow \mathbb{R}^{n_x}$ captures the discrete time system dynamics and
$\hat{x}_1$ is the initial state. We require that $\ell, f, c$ are twice continuously differentiable.

\subsection{Optimality Conditions}\label{ssec:kkteq}

The Lagrangian of the OCP in~\eqref{eq:oc} is given by
\begin{equation}\label{eq:lagrangian}
    \begin{aligned}
    \mathcal{L}(\mathbf{w}, \bm{\lambda}) \coloneqq  \lambda_1^\top (\hat{x}_1 & - x_1)  + \textstyle\sum_{t=1}^{N} L(x_t, u_t, \phi_t) \\
    &  + \textstyle\sum_{t=1}^{N-1} \lambda_{t+1}^\top (f(x_t, u_t) - x_{t+1}) ,
    \end{aligned}
\end{equation}
where $\phi_{1:N}$ and $\bm{\lambda} \coloneqq \lambda_{1:N}$ are the dual variables for the equality and dynamics constraints in~\eqref{eq:oc}, respectively.  
Furthermore, $w_t = (x_t, u_t, \phi_t)$, $\mathbf{w}\coloneqq w_{1:N}$ and finally, $L(x_t, u_t, \phi_t) \coloneqq \ell(x_t, u_t) + \phi_t^\top c(x_t, u_t)$. 

The first-order optimality conditions for \eqref{eq:oc} (also known as Karush-Kuhn-Tucker (KKT) conditions), necessitate that trajectories $x_{1:N}, u_{1:N}$ can only be a local solution of \eqref{eq:oc} if there exist dual variables $\lambda_{1:N}$ and $\phi_{1:N}$ such that
\begin{subequations}\label{eq:kkt}
    \begin{gather}
        \begin{split}
            \nabla_{x_t} \mathcal{L}(\mathbf{w}, \bm{\lambda}) = L^t_x - \lambda_t^\top  + \lambda_{t+1}^\top f^t_x = 0, 
        \end{split}\label{eq:kktx} \\
        \nabla_{u_t} \mathcal{L}(\mathbf{w}, \bm{\lambda}) = L^t_u + \lambda_{t+1}^\top f^t_u = 0,
        \label{eq:kktu}\\
        \nabla_{\lambda_t} \mathcal{L}(\mathbf{w}, \bm{\lambda}) = \begin{cases}
            (\hat{x}_1 - x_t)^\top = 0 & t = 1\\
            (f(x_{t-1}, u_{t-1}) - x_{t})^\top  =  0 & t > 1
        \end{cases}, \label{eq:kktdyn}\\
        \nabla_{\phi_t} \mathcal{L}(\mathbf{w}, \bm{\lambda}) = c(x_t, u_t)^\top =  0,\label{eq:kktstage}
    \end{gather}
\end{subequations}
where $L_x^t, L_u^t, f_x^t, f_u^t$ are partial derivatives of $L$ and $f$ at $w_t$ and noting that for $t=N$, we implicitly set $\lambda_{t+1} = 0$ since there are no dynamics constraints for step $t=N$.

\section{Background and Derivation of DDP}\label{sec:backgroundddp}

We now provide some background on the derivation of our constrained differential dynamic programming algorithm. 

\subsection{Dynamic Programming for Optimal Control}\label{ssec:dpforocp}

Optimal control problems are amenable to being solved using \emph{dynamic programming} (DP), which relies on Bellman's principle of optimality \cite{bellman1965dynamic}. Concretely, for the OCP in \eqref{eq:oc}, the optimality principle at time step $t\in[N]$ yields
\begin{equation}\label{eq:dynamicprogramming} 
    \begin{array}{rl}
    V_t^\star(x_t) \coloneqq \underset{u_t}{\text{min}}   & \ell(x_t, u_t) + V^\star_{t+1} (f(x_t, u_t)) \\
    \text{s.t.} & c(x_t, u_t) = 0,
    \end{array}
\end{equation}
where $V_t^\star$ is the \emph{optimal value function} with boundary condition $V_{N+1}^\star(x) \coloneqq 0$. By duality, \eqref{eq:dynamicprogramming} is equivalent to
\begin{equation}\label{eq:dynamicprogrammingminmax}
    V_t^\star(x_t) \coloneqq \underset{u_t}{\text{min}} \, \underset{\phi_t}{\text{max}} \,   L(x_t, u_t, \phi_t) + V^\star_{t+1} (f(x_t, u_t)),
\end{equation}
implying that, 1) we can decompose \eqref{eq:oc} into a sequence of sub-problems, which are solved recursively backwards in time from $t=N$ and,  2) $V_t^\star(x_t)$ is the optimal value of the Lagrangian of the subsequence starting at step $t$ and state $x_t$.

\subsection{Differential Dynamic Programming}\label{ssec:backgroundddp}

A core idea used to derive the FilterDDP algorithm, analogous to \cite{pavlov_ipddp}, is to replace the intractable \emph{optimal} value function $V_t^\star$ with a tractable \emph{sub-optimal} $V^t$. As a result, the intractable problems defined in \eqref{eq:dynamicprogrammingminmax} are replaced with
\begin{equation}\label{eq:ipddpminmax}
    \underset{u_t}{\text{min}}  \, \underset{\phi_t}{\text{max}} \,   Q^t(x_t, u_t, \phi_t),
\end{equation}
where $Q^t(x_t, u_t, \phi_t) \coloneqq L(x_t, u_t, \phi_t)  + V^{t+1} (f(x_t, u_t))$. 

The FilterDDP algorithm alternates between a backward pass and forward pass phase as in unconstrained DDP \cite{mayneddpbook}. The backward pass applies a \emph{perturbed} Newton step to \eqref{eq:ipddpminmax} to determine the nonlinear update rule applied in the forward pass. We defer formally defining $V^t$ to Sec. \ref{ssec:value}.

\section{Filter Differential Dynamic Programming}\label{sec:filterddp}

We now describe the full FilterDDP algorithm. FilterDDP generates a sequence of iterates $\{(\mathbf{w}_k, \bm{\lambda}_k)\}$ by alternating between a \textit{backward pass} and \textit{forward pass}.

\subsection{Termination Criterion}\label{ssec:termination}

Let $w_{t} \coloneqq (x_{t}, u_{t}, \phi_{t})$ and $\lambda_t$ together, represent the current iterate at time step $t$. The optimality error used to determine convergence, letting $c^{t} \coloneqq c(x_{t}, u_{t})$, is defined to be
\begin{equation}\label{eq:opterr}
    \begin{aligned}
    E(\mathbf{w}, \bm{\lambda}) \coloneqq \max_t \max \big\{ \|\nabla_{u_t}\mathcal{L}(\mathbf{w}, \, \bm{\lambda}) \|_\infty,
    \|c^t\|_\infty \big\}.
    \end{aligned}
\end{equation}
The algorithm terminates successfully if $E(\mathbf{w}, \bm{\lambda}) <
\epsilon_\text{tol}$ for a user-defined tolerance $\epsilon_\text{tol} > 0$. We omit $\nabla_{x_t}\mathcal{L}$ in \eqref{eq:opterr} since for the iterates under the FilterDDP algorithm, it will be 0 under the definition of $\lambda_t$ to be presented in \eqref{eq:vxupdate}.

\subsection{Backward Pass}\label{ssec:backward}

The backward pass computes the first and perturbed second derivatives of $V^t$ around the current iterate $\bar{\mathbf{w}}$, denoted by $\bar{V}^t_x\in\mathbb{R}^{1\times n_x}$ and $P_t\in\mathbb{R}^{n_x \times n_x}$, backwards in time. Next, an update rule is derived using a \emph{perturbed} Newton step applied to the stationarity condition of \eqref{eq:ipddpminmax}. The perturbation allows us to establish local convergence of FilterDDP in Sec. \ref{sec:proofconvergence}.

\paragraph*{Notation} For a function $h \in \{\ell, f, c, L, Q^t\}$ and $y \in \{x, u\}$, let $\bar{h}^{t}_{y} \coloneqq \nabla_{y} h(\bar{x}_t, \bar{u}_t, \bar{\phi}_t)$. We use similar notation for second derivatives of $h$. Furthermore, let $\bar{h}^{t} \coloneqq h(\bar{x}_t, \bar{u}_t, \bar{\phi}_t)$.

The backward pass sets the boundary conditions
\begin{equation}\label{eq:vxboundary}
     \bar{\lambda}_{N+1} = 0_{1 \times n_x}, \,\,\, \bar{V}_x^{N+1} = 0_{1\times n_x}, \,\,\, P_{N+1} = 0_{n_x \times n_x},
\end{equation}
and proceeds backwards in time, alternating between the perturbed Newton step and updating $\bar{V}^t_x$, $P_t$ and $\bar{\lambda}_t$.

\paragraph{Update Rule}
We first present the unperturbed Newton step to the stationarity condition of \eqref{eq:ipddpminmax}, given by
\begin{equation}\label{eq:ipddpminmaxkkt}
    \nabla_{u_t} Q^t(x_t, u_t, \phi_t) = 0, \quad c(x_t, u_t) = 0.
\end{equation}
\pagebreak
Let $\delta w_t \coloneqq (\delta x_t, \delta u_t, \delta \phi_t)$ represent an update to $\bar{w}_t$. Newton's method applied to \eqref{eq:ipddpminmaxkkt} yields
\begin{equation}\label{eq:newtonminmax1}
    \begin{bmatrix}
        \bar{Q}_{uu}^t & A_t  \\
        A_t^\top & 0
    \end{bmatrix}
    \begin{bmatrix}
        \delta u_t \\ \delta \phi_t
    \end{bmatrix}
    = - \begin{bmatrix}
        (\bar{Q}_u^t)^\top \\ \bar{c}^t 
    \end{bmatrix} -
    \begin{bmatrix}
        \bar{Q}_{ux}^t \\ \bar{c}_x^t
    \end{bmatrix} \delta x_t,
\end{equation}
where $A_t \coloneqq (\bar{c}_u^t)^\top$, $\bar{Q}_u^t = \bar{L}_u^t + \bar{V}_{x}^{t+1}\bar{f}_u^t$,
\begin{subequations}\label{eq:qfnderivs}
    \begin{gather}
    \bar{Q}_{uu}^t = \bar{L}_{uu}^t + (\bar{f}_u^t)^\top \bar{V}_{xx}^{t+1} \bar{f}_u^t + \bar{V}^{t+1}_{x} \cdot \bar{f}_{uu}^t, \label{eq:Quut}\\
    \bar{Q}_{ux}^t = \bar{L}_{ux}^t + (\bar{f}_u^t)^\top \bar{V}_{xx}^{t+1} \bar{f}_x^t + \bar{V}^{t+1}_{x} \cdot \bar{f}_{ux}^t, \label{eq:Quxt} \\
    \bar{Q}_{xx}^t = \bar{L}_{xx}^t + (\bar{f}_x^t)^\top \bar{V}_{xx}^{t+1} \bar{f}_x^t + \bar{V}^{t+1}_{x} \cdot \bar{f}_{xx}^t, \label{eq:Qxxt}
    \end{gather}
\end{subequations}
and $\cdot$ denotes a tensor contraction along the first dimension. Equation \eqref{eq:newtonminmax1} implies that $\delta u_t = \alpha_t + \beta_t \delta x_t$ and $\delta \phi_t = \psi_t + \omega_t \delta x_t$ for some parameters $\alpha_{t}, \beta_{t}$, $\psi_{t}$, $ \omega_{t}$. To determine the update rule, FilterDDP applies the perturbed Newton step given by
\begin{equation}\label{eq:bwkktfull}
    \underbrace{\begin{bmatrix}
        H_t + \delta_w I & A_t \\
        A_t^\top & -\delta_c I
    \end{bmatrix}}_{K_t}
    \begin{bmatrix}
        \alpha_t & \beta_t \\ \psi_t & \omega_t
    \end{bmatrix}
     = -\begin{bmatrix}
         (\bar{Q}_u^t)^\top & B_t \\ \bar{c}^t & \bar{c}_x^t
     \end{bmatrix},
\end{equation}
where $H_t$, $B_t$, $C_t$ are given by
\begin{subequations}\label{eq:bwhessperturbed}
    \begin{align}
    H_t &\coloneqq \bar{L}_{uu}^t + (\bar{f}_u^t)^\top P_{t+1} \bar{f}_u^t + \bar{\lambda}_{t+1} \cdot \bar{f}_{uu}^t, \label{eq:Ht} \\
    B_t &\coloneqq \bar{L}_{ux}^t + (\bar{f}_u^t)^\top P_{t+1} \bar{f}_x^t + \bar{\lambda}_{t+1} \cdot \bar{f}_{ux}^t, \label{eq:Bt} \\
    C_t &\coloneqq \bar{L}_{xx}^t + (\bar{f}_x^t)^\top P_{t+1} \bar{f}_x^t + \bar{\lambda}_{t+1} \cdot \bar{f}_{xx}^t \label{eq:Ct}
    \end{align}
\end{subequations}
and the recursive update of $P_t$ is described in \eqref{eq:vxxupdate}.

Motivated by global convergence \cite{wachterglobal}, we set $\delta_w, \delta_c \geq 0$ following Alg. IC in \cite{ipopt} so that $K_t$ has an inertia\footnote{The inertia is the number of positive, negative and zero eigenvalues.} of $(n_u, n_c, 0)$. We factorize $K_t$ and recover its inertia using an LDLT factorization \cite{bunchkaufman, rookpivot}.

\paragraph{Updating the Value Function Approximation}

After solving \eqref{eq:bwkktfull}, $\bar{V}_x^t, P_t$ and $\bar{\lambda}_t$ are updated according to 
\begin{subequations}\label{eq:valueddpupdate}
    \begin{gather}
    \bar{V}_x^t = \bar{Q}_x^t + \bar{Q}_u^t \beta_t + (\bar{c}^t)^\top\omega_t, \quad \bar{\lambda}_{t}^\top = \bar{L}_x^t + \bar{\lambda}^\top_{t+1} \bar{f}_x^t\label{eq:vxupdate} \\
    P_t = C_t + \beta_t^\top H_t \beta_t + B_t^\top \beta_t + \beta_t^\top B_t. \label{eq:vxxupdate}
    \end{gather}
\end{subequations}
The time step is then decremented, i.e., $t \gets t-1$ and \emph{a)} is repeated. We will defer defining the value function $V^t$ from which $\bar{V}_x^t$ and $\bar{V}_{xx}^t$ are derived until Sec. \ref{ssec:value}, since the forward pass in Sec. \ref{ssec:forward} must be described first.

\paragraph{Differences to prior works e.g., \cite{proxddp, pavlov_ipddp, altro}} Both the perturbed Newton step \eqref{eq:bwkktfull} and replacing $\|\bar{Q}_u^t\|_\infty$ with \eqref{eq:opterr} for the termination criteria are novel inclusions of FilterDDP, motivated by the convergence result in Sec. \ref{sec:proofconvergence}.

\subsection{Forward Pass}\label{ssec:forward}

The forward pass generates trial points for the next iterate. For a step size $\gamma \in (0, 1]$, the trial point is given by applying
\begin{equation}\label{eq:forwardpass}
    \begin{gathered}
        x^+_{t+1} \coloneqq f(x^+_t, u^+_t(x_t^+, \gamma)), \quad x^+_1 = \hat{x}_1,\\
        u_t^+(x_t^+, \gamma) \coloneqq \bar{u}_t + \gamma \alpha_t + \beta_t(x^+_t - \bar{x}_t), \\
        \phi_t^+(x_t^+, \gamma) \coloneqq \bar{\phi}_t + \gamma\psi_t + \omega_t (x^+_t - \bar{x}_t),
    \end{gathered}
\end{equation}
forward in time starting from $t=1$. A trial point $\mathbf{w}^+(\gamma)$, where $w_t^+(\gamma) \coloneqq (x_t^+, u_t^+(x_t^+, \gamma), \phi_t^+(x_t^+, \gamma))$ is accepted as the next iterate if step acceptance criteria described below are satisfied. The backtracking line search procedure sets $\gamma = 1$, setting $\gamma \gets \frac{1}{2}\gamma$ if the trial point $\mathbf{w}^+(\gamma)$ is not accepted.


\emph{\,\, b) Step acceptance:} A trial point must yield a sufficient decrease in \emph{either} the Lagrangian of \eqref{eq:oc} or the constraint violation compared to the current iterate, concretely,
\begin{equation}\label{eq:filtercriteria}
        \mathcal{L}(\mathbf{w}) \coloneqq  \sum_{t=1}^{N} L(x_t, u_t, \phi_t), \quad \theta(\mathbf{w}) \coloneqq \sum_{t=1}^N \|c^t\|_1.
\end{equation}
We adopt the sufficient decrease condition given by
\begin{subequations}\label{eq:sufficientfilter}
    \begin{gather}
        \theta(\mathbf{w}^+(\gamma)) \leq (1 - \gamma_\theta) \theta(\bar{\mathbf{w}}) \quad \text{or} \label{eq:suffdecreaseconstr}\\
            \mathcal{L}(\mathbf{w}^+(\gamma)) \leq \mathcal{L}(\bar{\mathbf{w}}) - \gamma_\mathcal{L} \theta(\bar{\mathbf{w}}),
        \label{eq:suffdecreasecost}
    \end{gather}
\end{subequations}
for constants $\gamma_\theta, \gamma_\mathcal{L} \in (0, 1)$. Furthermore, following \cite{wachterglobal, ipopt}, \eqref{eq:sufficientfilter} is replaced by enforcing a sufficient decrease condition on $\mathcal{L}$ if both $\theta(\bar{\mathbf{w}}) < \theta_\mathrm{min}$ for some small $\theta_\mathrm{min}$ and the \emph{switching condition}
\begin{equation}\label{eq:switching}
    \gamma m < 0 \quad \text{and} \quad (-\gamma m)^{s_\mathcal{L}}\gamma^{1 - s_\mathcal{L}} > \delta (\theta(\bar{\mathbf{w}}))^{s_\theta}
\end{equation}
holds, where $\delta > 0$, $s_\theta > 1$, $s_\mathcal{L} \geq 1$ are constants and
\begin{equation}\label{eq:expecteddecrease}
    \begin{aligned}
        m &\coloneqq \sum_{t=1}^{N} \left(\bar{Q}_u^t \alpha_t  + \psi_t^\top \bar{c}^t \right).
    \end{aligned}
\end{equation}
The switching condition \eqref{eq:switching} is motivated by the global convergence analysis of \cite{wachterglobal}. The aforementioned decrease condition on $\mathcal{L}$ is given by 
\begin{equation}\label{eq:armijo}
    \mathcal{L}(\mathbf{w}^+(\gamma)) \leq \mathcal{L}(\bar{\mathbf{w}}) + \eta_\mathcal{L} \gamma m
\end{equation}
for some small $\eta_\mathcal{L} > 0$. Accepted iterates which satisfy~\eqref{eq:switching} and~\eqref{eq:armijo} are called $\mathcal{L}$-type iterations~\cite{ipopt}. 

FilterDDP maintains a \emph{filter}, denoted by $\mathcal{F} \subseteq \{(\theta, \mathcal{L}) \in \mathbb{R}^2 : \theta \geq 0\}$. The set $\mathcal{F}$ defines a ``taboo" region for iterates, and a trial point is rejected if $(\theta^+,  \mathcal{L}^+) \in \mathcal{F}$. We initialise the filter with $\mathcal{F} = \{(\theta,  \mathcal{L}) \in \mathbb{R}^2 : \theta \geq \theta_{\mathrm{max}}\}$ for some $\theta_\mathrm{max}$. After a trial point $\mathbf{w}^+(\gamma)$ is accepted, the filter is augmented if either~\eqref{eq:switching} or~\eqref{eq:armijo} do not hold, using
\begin{equation}\label{eq:filterupdate}
    \begin{gathered}
        \mathcal{F}^+ \coloneqq \mathcal{F} \cup \big\{  (\theta,  \mathcal{L}) \in \mathbb{R}^2 : \theta \geq (1 - \gamma_\theta) \theta(\bar{\mathbf{w}}),\\ 
     \quad\quad\quad \mathcal{L} \geq  \mathcal{L}(\bar{\mathbf{w}}) - \gamma_\mathcal{L} \theta(\bar{\mathbf{w}}) \big\}.
    \end{gathered}
\end{equation}
\paragraph{Differences to \cite{pavlov_ipddp}, \cite{mastallinullspace}} In contrast to existing non-AL DDP methods, FilterDDP uses the \emph{Lagrangian} within the step acceptance criteria instead of the cost. We motivate this by formally showing in Sec. \ref{ssec:value} that \eqref{eq:armijo} is analogous to the well-known Armijo condition \cite[pg. 33]{nwopt}
. We note that global convergence of line-search filter methods for general nonlinear programs is preserved under this change \cite[Sec. 4.1]{wachterglobal}.

\subsection{Complete FilterDDP Algorithm}
Below, we present the full FilterDDP algorithm.

\textbf{Algorithm 1:}

\textit{Given:} Starting point $\mathbf{w}_0$; $\theta_\mathrm{max}\in (0, \infty]$; $\gamma_\theta, \gamma_\mathcal{L} \in (0, 1)$; $\epsilon_\mathrm{tol}, \delta, \gamma^{\min}, \text{max\_iters} > 0$; $s_\theta > 1$; $s_\mathcal{L} \geq 1$; $\eta_\mathcal{L} \in (0, \frac{1}{2})$.
\begin{enumerate}[label*=\arabic*.]
    \item \textit{Initialise.} \,\, Initialise the iteration counter $k \gets 0$ and filter $\mathcal{F}_k \coloneqq \left\{ (\theta, \mathcal{L}) \in \mathbb{R}^2 : \theta \geq \theta_{\max} \right\}$.
    \item \label{alg:backward} \textit{Backward pass.} \,\, Set value function boundary conditions, i.e., \eqref{eq:vxboundary} and set time counter $t \gets N$.
    \begin{enumerate}[label*=\arabic*.]
        \item If $t = 0$, go to step \ref{alg:checkconverge} \label{alg:beginbw}
        \item Compute blocks in \eqref{eq:bwkktfull}, i.e., $H_t$, $B_t$, etc. using $\mathbf{w}_k$, $\bm{\lambda}_k$ and solve for $\alpha_{t}, \beta_{t}, \psi_{t}$, $\omega_{t}$. If the matrix in \eqref{eq:bwkktfull} is too ill conditioned, terminate the algorithm. \label{alg:restofrombw}
        \item Update $\bar{V}_x^{t}$, $P_{t}$, $\bar{\lambda}_{t}$ using \eqref{eq:valueddpupdate}.
        \item Set $t \gets t - 1$ and go to \ref{alg:beginbw}
    \end{enumerate}
    \item \textit{Check convergence.} \,\, If $E(\mathbf{w}_{k}$, $\bm{\lambda}_k) < \epsilon_\mathrm{tol}$ then STOP (success). If $k = \text{max\_iters}$ then STOP (failure). \label{alg:checkconverge}
    \item \label{alg:backtrackingls} \textit{Backtracking line search.} \,\, 
        \begin{enumerate}[label*=\arabic*.]
            \item \textit{Initialise line search.} \,\, Set $l \gets 0$ and $\gamma_{l} = 1$. \label{alg:commencels}
            \item \textit{Compute new trial point.} \,\, If the trial step size $\gamma_{l}$ becomes too small, i.e., $\gamma_{l} < \gamma^{\min}$, terminate the algorithm. Otherwise, compute the new trial point $\mathbf{w}^+(\gamma_{l})$ using the forward pass, i.e., \eqref{eq:forwardpass}.  \label{alg:propnewtrialpoint}
            \item \textit{Check acceptability to the filter.} \,\, If \eqref{eq:sufficientfilter} does not hold, reject the trial step size and go to step \ref{alg:newstepsize}
            \item \textit{Check sufficient decrease against the current iterate.} \,\,
                \begin{enumerate}[label*=\arabic*.]
                    \item \textit{Case} I: $\gamma_{l}$ is a $\mathcal{L}$-step size (i.e., \eqref{eq:switching} holds). If the Armijo condition \eqref{eq:armijo} holds, accept the trial step and go to step \ref{alg:accepttrial} Otherwise, go to step \ref{alg:newstepsize}
                    \item \textit{Case} II: $\gamma_{l}$ is not a $\mathcal{L}$-step size, (i.e., \eqref{eq:switching} is not satisfied): If \eqref{eq:sufficientfilter} holds, accept the trial step and go to step \ref{alg:accepttrial} Otherwise, go to step \ref{alg:newstepsize}
                \end{enumerate}
        \end{enumerate}
        \item \textit{Choose new trial step size.} Choose $\gamma_{l+1} = \frac{1}{2} \gamma_{l}$, set $l \gets l+1$ and go back to step \ref{alg:propnewtrialpoint} \label{alg:newstepsize}
        \item \textit{Accept trial point.} Set $\gamma \coloneqq \gamma_{l}$ and $\mathbf{w}_{k+1} \coloneqq \mathbf{w}_{k}^+(\gamma)$. \label{alg:accepttrial}
        \item \textit{Augment filter if necessary.} \,\, If $k$ is not a $\mathcal{L}$-type iteration, augment the filter using \eqref{eq:filterupdate}; otherwise leave the filter unchanged, i.e., set $\mathcal{F}_{k+1} \coloneqq \mathcal{F}_k$. 
        \item \textit{Continue with next iteration.} \,\, Increase the iteration counter $k \gets k+1$ and go back to step \ref{alg:backward} \label{alg:itercounterincre}
\end{enumerate}

\subsection{Definition of the Sub-Optimal Value Function $V^t$}\label{ssec:value}

We conclude the section by presenting a convenient candidate for the value function $V^t$ whose derivatives satisfies \eqref{eq:valueddpupdate}. In doing so, we also provide intuition on \eqref{eq:valueddpupdate} and \eqref{eq:armijo}.

\begin{theorem}
    A value function $V^t$ which has associated first and second-order derivatives given by \eqref{eq:vxupdate} and $\bar{V}_{xx}^t = \bar{Q}_{xx}^t + \beta_t^\top \bar{Q}_{uu}^t \beta_t + \bar{Q}_{xu}^t \beta_t + \beta_t^\top \bar{Q}_{ux}^t$ (i.e., unperturbed \eqref{eq:vxxupdate}), respectively, is given by $V^{N+1}(x, \gamma) \coloneqq 0$, and
    \begin{equation}\label{eq:truevalue}
        V^t(x, \gamma) \coloneqq Q^t(x, u_t^+(x, \gamma) , \phi_t^+(x, \gamma)).
\end{equation}
Note that $V^1(\hat{x}_1, \gamma) = \mathcal{L}(\mathbf{w}^+(\gamma))$, i.e., the updated $\mathcal{L}$, and that $\bar{V}_x^t = \nabla_x V^t(\bar{x}_t, 0)$ and $\bar{V}_{xx}^t = \nabla_{xx}^2 V^t(\bar{x}_t, 0)$.
\end{theorem}
    
\begin{proof}
    For the first derivative,
    \begin{equation*}
    \begin{aligned}
        \nabla_x V^t(\bar{x}_t, 0) & \overset{\eqref{eq:truevalue}}{=} \bar{Q}^t_x + \bar{Q}_u^t \nabla_x u^+_t(\bar{x}_t, 0) + \bar{Q}^t_\phi \phi_t^+(\bar{x}_t, 0) \\\
        &\overset{\eqref{eq:forwardpass}}{=} \bar{Q}_x^t +\bar{Q}_u^t \beta_t + (\bar{c}^t)^\top \omega_t = \bar{V}_x^t.
    \end{aligned}
    \end{equation*}
    
    For the second deriative, let $y(x, \gamma)\coloneqq (x, u_t^+(x, \gamma), \phi_t^+(x, \gamma))$. It follows that
    
    \begin{equation*}
        \begin{aligned}
            \nabla_{xx}^2 V^t(\bar{x}_t, 0) &= \nabla_{xx}^2 Q^t(y(\bar{x}_t, 0)) \\
            &= (\nabla_x y(\bar{x}_t, 0))^\top \nabla_{yy}^2 Q^t(y(\bar{x}_t, 0)) \nabla_x y(\bar{x}_t, 0) \\
            & \quad \quad + \nabla_y Q^t(y(\bar{x}_t, 0)) \cdot \underbrace{\nabla_{xx}^2 y(\bar{x}_t, 0)}_{=0} \\
            &= \begin{bmatrix}
                I & \beta_t^\top & \omega_t^\top
            \end{bmatrix}
            \begin{bmatrix}
                \bar{Q}^t_{xx} & \bar{Q}_{xu}^t & (\bar{c}_x^t)^\top \\
                \bar{Q}_{ux}^t & \bar{Q}_{uu}^t & A_t \\
                \bar{c}_x^t & A_t^\top & 0
            \end{bmatrix}
            \begin{bmatrix}
                I \\ \beta_t \\ \omega_t
            \end{bmatrix} \\
            &= \bar{Q}_{xx}^t + \bar{Q}_{xu}^t \beta_t + (\bar{c}_x^t)^\top \omega_t + \beta_t^\top \bar{Q}_{ux}^t \\
            &\quad + \beta_t^\top \bar{Q}_{uu}^t \beta_t + \beta_t^\top A_t \omega_t + \omega_t^\top \bar{c}_x^t + \omega_t^\top A_t^\top \beta_t  \\
            &\overset{\eqref{eq:bwkktfull}}{=} \bar{Q}_{xx}^t + \beta_t^\top \bar{Q}_{uu}^t \beta_t + \bar{Q}_{ux}^t \beta_t + \beta_t^\top \bar{Q}_{ux}^t\\
            &= \bar{V}_{xx}^t,
        \end{aligned}
    \end{equation*}
    where \eqref{eq:bwkktfull} is used through the relation $A_t^\top \beta_t = -\bar{c}_x^t$.
\end{proof}

\begin{corollary}\label{cor:armijo}
    $m = \frac{d}{d\gamma} \mathcal{L}(\mathbf{w}^+(0))$, i.e., $m$ is similar to the directional derivative in the Armijo condition \cite[pg. 33]{nwopt}.
\end{corollary}
\begin{proof}
    Differentiating \eqref{eq:truevalue} w.r.t. $\gamma$ and evaluating at $\gamma = 0$ and $x = \bar{x}$ yields $\bar{V}_\gamma^t = \bar{Q}_u^t \alpha_t + \psi_t^\top \bar{c}^t + \bar{V}_\gamma^{t+1}$. Applying an inductive argument yields $\bar{V}^1_\gamma = m$. Since $V^1(\hat{x}_1, \gamma) = \mathcal{L}(\mathbf{w}^+(\gamma))$, it follows that $\bar{V}_\gamma^1 = \frac{d}{d\gamma} \mathcal{L}(\mathbf{w}^+(0))$.
\end{proof}

Corollary \ref{cor:armijo} motivates using $\mathcal{L}$ for step acceptance in \eqref{eq:armijo}.

\section{Convergence Analysis of FilterDDP}\label{sec:proofconvergence}

In this section, we establish the local quadratic convergence of FilterDDP within a sufficiently small neighbourhood of a solution. We simplify our analysis by assuming all iterates with full step size $\gamma = 1$ are accepted by the filter criteria in Sec. \ref{ssec:forward}, leaving the more general analysis (with a second-order correction step \cite{wachterlocal}) for future work. We first establish that an optimal point $\mathbf{w}^\star$ satisfies KKT conditions \eqref{eq:kkt} if and only if it is a fixed point of FilterDDP. We then prove local quadratic convergence to $\mathbf{w}^\star$.

\textbf{Assumptions 1.} \,\, For all $t\in[N]$ there exists an open convex neighbourhood $\mathcal{N}_t$ containing $w_t^\star$ and the current iterate $\bar{w}_t$, such that for all $w_t \in \mathcal{N}_t$ and all $t\in[N]$,

\begin{enumerate}[label=(1\Alph*)]
    \item functions $\ell, f, c$ and their first derivatives are Lipschitz continuously differentiable; \label{asm:lipschitz}
    \item the minimum singular value of $A_t$ is uniformly bounded away from zero; and \label{asm:Atfullrank}
    \item $\xi^\top H_t\xi > 0$ for all $\xi$ in the null space of $A_t$. \label{asm:pdhessian}
\end{enumerate}

Assumption \ref{asm:pdhessian} ensures that no inertia correction is applied and that the iteration matrix $K_t$ is non-singular. Furthermore, Assumption \ref{asm:Atfullrank} holds under suitable constraint qualifications on \eqref{eq:oc} (e.g., LICQ \cite[Definition 12.4]{nwopt}).

\begin{theorem}\label{thm:stationary}
    Let $\bar{\mathbf{w}}$ be a feasible point which satisfies $\bar{x}_{t+1} = f(\bar{x}_t, \bar{u}_t)$, $\bar{x}_1 = \hat{x}_1$ and $c(\bar{x}_t, \bar{u}_t) = 0$. Then $\bar{\mathbf{w}}$ is a fixed point of FilterDDP (i.e., $\alpha_t=0$, $\psi_t = 0$ for all $t\in[N]$) if and only if it also satisfies KKT conditions \eqref{eq:kkt}.
\end{theorem}

\begin{proof}
    Suppose $\bar{\mathbf{w}}$ is a fixed point of FilterDDP. Then by non-singularity of $K_t$, $\bar{Q}_u^t = 0$, $\bar{c}^t = 0$. By \eqref{eq:vxupdate} and the definition of $\bar{Q}_u^t$, $\bar{V}_x^t = \bar{\lambda}_t$ for all $t\in[N]$. It follows that $\nabla_{u_t} \mathcal{L}(\bar{\mathbf{w}}, \bar{\bm{\lambda}}) = \bar{L}_u^t + \bar{\lambda}_{t+1} \bar{f}_u^t = \bar{Q}_u^t = 0$. Substituting the update rule for $\bar{\lambda}_t$ into \eqref{eq:kktx} yields $\nabla_{x_t} \mathcal{L}(\bar{\mathbf{w}}, \bar{\bm{\lambda}}) = 0$. The remaining conditions in \eqref{eq:kkt} are satisfied by assumption.

    Now suppose that $\bar{\mathbf{w}}$ satisfies \eqref{eq:kkt}. Then $\nabla_{u_t} \mathcal{L}(\bar{\mathbf{w}}, \bar{\bm{\lambda}}) = \bar{L}_u^t + \bar{\lambda}_{t+1}^\top \bar{f}_u^t = 0$ and $\bar{c}^t = 0$. Since $\bar{c}^N = 0$ and $\bar{L}_u^N = 0$, by \eqref{eq:vxupdate} it follows that $\bar{V}_x^N = \bar{\lambda}_N$. Since $\bar{c}^{N-1} = 0$ and $\bar{L}_u^{N-1} + \bar{\lambda}_{N}^\top \bar{f}_u^{N-1} = 0$, by \eqref{eq:vxupdate} it follows that $\bar{V}_x^{N-1} = \bar{\lambda}_{N-1}$ and consequently $\bar{Q}_u^{N-1} = 0$. An inductive argument yields $\bar{Q}_u^t = 0$ for all $t\in[N]$, which with $\bar{c}^t = 0$ and $K_t$ non-singular implies that $\bar{\mathbf{w}}$ is a fixed point of FilterDDP.
\end{proof}

We now establish local quadratic convergence. Denote an optimal point by $\mathbf{w}^\star$ and let $R$ abbreviate $O(\|\mathbf{w}^\star - \bar{\mathbf{w}}\|^2)$. 

\begin{proposition}\label{prop:quadtaylor}
    Let $\lambda^{N+1}(\mathbf{w}) = 0_{n_x}$ and for $t\in[N]$,
    \begin{equation}\label{eq:lambdaw}
        \lambda^t(\mathbf{w})^\top \coloneqq \nabla_x L(x_t, u_t, \phi_t) +  \lambda^{t+1}(\mathbf{w})^\top \nabla_x f(x_t, u_t).
    \end{equation}
    Note that $ \lambda^t(\bar{\mathbf{w}}) = (\bar{\lambda}_t)^\top$. Under Assumptions 1,
    \begin{equation}\label{eq:valuelambda}
        \lambda^t(\mathbf{w}^\star) = (\bar{V}_x^t)^\top + P_t (x_t^\star - \bar{x}_t) + R.
    \end{equation}
\end{proposition}

\begin{proof}
    We proceed with an inductive argument. The base case $t=N+1$ holds since $\lambda^{N+1} = 0$, $\bar{V}_x^{N+1} = 0$ and $P_{N+1} = 0$. Suppose \eqref{eq:valuelambda} holds for some step $t+1$. By Ass. \ref{asm:lipschitz}, we can apply Taylor's theorem to $f$ and $c$, yielding
    \begin{equation}\label{eq:dyntaylor}
        \begin{aligned}
            x_{t+1}^\star &= \bar{x}_{t+1} + \bar{f}_x^t (x_t^\star - \bar{x}_t) + \bar{f}_u^t(u_t^\star - \bar{u}_t) + R,
        \end{aligned}
    \end{equation}
    \begin{equation}\label{eq:kktpritaylor}
        0=c^t(x_t^\star, u_t^\star) = \bar{c}^t + A_t^\top (u_t^\star - \bar{u}_t) + \bar{c}_x^t (x_t^\star - \bar{x}_t) + R.
    \end{equation}
    In addition, by Ass. \ref{asm:lipschitz}, the derivatives of $\lambda^t$ exist (locally) and are Lipschitz, given by $\bar{\lambda}_{w_{t+1:N}}^t = (\bar{f}_{x}^t)^\top \bar{\lambda}_{w_{t+1:N}}^{t+1}$ and
    \begin{equation}\label{eq:lambdaderivt}
        \bar{\lambda}_{y_t}^t = \bar{L}_{xy}^t + \bar{\lambda}^{t+1} \cdot \bar{f}_{xy}^t, \,\,\,\, y\in\{x, u\}, \,\,\,\, \bar{\lambda}_\phi^t = (\bar{c}_x^t)^\top,
    \end{equation}
    also noting that derivatives of $\lambda^t$ w.r.t $w_s$ for $s < t$ are zero.

   Applying Taylor's theorem to \eqref{eq:lambdaw} for step $t$ yields   
    \begin{equation}\label{eq:lambdataylor}
    \setlength\arraycolsep{0pt}
    \begin{array}{ccl}
        \lambda^t(\mathbf{w}^\star) 
        & = & \bar{\lambda}^{t} + \bar{\lambda}^t_{\mathbf{w}} (\mathbf{w}^\star - \bar{\mathbf{w}}) + R \\
        & \overset{\eqref{eq:lambdaw}}{=} & \bar{\lambda}^t + \bar{\lambda}_{w_t}^t (w_t^\star - \bar{w}_t)  \\
        & & \quad + (\bar{f}_x^t)^\top \bar{\lambda}_{\mathbf{w}}^{t+1}(\mathbf{w}^\star -  \bar{\mathbf{w}}) + R\\
        & \overset{\eqref{eq:lambdaw}}{=} & (\bar{L}^t_x)^\top + \bar{\lambda}_{w_t}^t (w_t^\star - \bar{w}_t) \\
        & & \quad + (\bar{f}_x^t)^\top \lambda^{t+1}(\mathbf{w}^\star) + R \\
        & \overset{\eqref{eq:valuelambda}, \eqref{eq:dyntaylor}}{=} & (\bar{Q}^t_x)^\top + \bar{\lambda}_{w_t}^t (w_t^\star - \bar{w}_t) + (\bar{f}_x^t)^\top P_{t+1} (  \\
        & & \quad \bar{f}_x^t (x_{t}^\star - \bar{x}_{t}) + \bar{f}_u^t (u_{t}^\star - \bar{u}_{t})) + R \\
        & \overset{\eqref{eq:lambdaderivt}, \eqref{eq:bwhessperturbed}}{=} & (\bar{Q}_x^t)^\top + C_t(x_t^\star - \bar{x}_t) + B_t^\top (u_t^\star - \bar{u}_t) \\
        & & \quad + (\bar{c}_x^t)^\top (\phi_t^\star - \bar{\phi}_t) + R.\\

    \end{array}
    \end{equation}
    By Theorem \ref{thm:stationary}, $\nabla_{u_t} \mathcal{L}(\mathbf{w}^\star, \bm{\lambda}(\mathbf{w}^\star)) = 0$. Applying Taylor's theorem to this expression around $\bar{\mathbf{w}}$ yields,
    \begin{equation}\label{eq:ghattaylor}
        \setlength\arraycolsep{0pt}
        \begin{array}{ccl}
        0 & = & \bar{\mathcal{L}}_{u_t}^t + \bar{\mathcal{L}}_{u_t w_t}^t (w_t^\star - \bar{w}_t) \\
        & & \quad + \bar{\mathcal{L}}_{u_t w_{t+1:N}} (w_{t+1:N}^\star - \bar{w}_{t+1:N}) + R \\
        & \overset{\eqref{eq:kktu}}{=} & (\bar{L}_u^t)^\top + (\bar{f}_u^t)^\top \bar{\lambda}_{t+1} + \bar{\mathcal{L}}_{u_t w_t}^t (w_t^\star - \bar{w}_t) \\
        & & + (\bar{f}_u^t)^\top \bigl( \bar{\lambda}^{t+1}_{w_{t+1:N}}(w_{t+1:N}^\star - \bar{w}_{t+1:N})\bigr) + R\\
        & \overset{\eqref{eq:valuelambda}, \eqref{eq:dyntaylor}}{=} & (\bar{Q}_u^t)^\top + \bar{\mathcal{L}}_{u_t w_t}^t (w_t^\star - \bar{w}_t) + (\bar{f}_u^t)^\top P_{t+1} (  \\
        & & \quad \bar{f}_x^t (x_{t}^\star - \bar{x}_{t}) + \bar{f}_u^t (u_{t}^\star - \bar{u}_{t})) + R\\

        & \overset{\eqref{eq:kktu}}{=} & R + (\bar{Q}_u^t)^\top + A_t(\phi_t^\star - \bar{\phi}_t) \\
        & & + (\bar{L}_{uu}^t + (\bar{f}_u^t)^\top P_{t+1} \bar{f}_u^t + \bar{\lambda}_{t+1} \cdot \bar{f}_{uu}^t) (u_t^\star - \bar{u}_t)   \\\
        & & + (\bar{L}_{ux}^t + (\bar{f}_u^t)^\top P_{t+1} \bar{f}_x^t + \bar{\lambda}_{t+1} \cdot \bar{f}_{ux}^t)(x_t^\star - \bar{x}_t) \\
        
        & \overset{\eqref{eq:bwhessperturbed}}{=} & (\bar{Q}_u^t)^\top + H_t(u_t^\star - \bar{u}_t) + B_t (x_t^\star - \bar{x}_t) 
        \\
        & & \quad + A_t(\phi_t^\star - \bar{\phi}_t) + R.
        \end{array}
    \end{equation}

    Rearranging \eqref{eq:ghattaylor} for $(\bar{Q}_u^t)^\top$, substituting the resulting expression into \eqref{eq:vxupdate} and substituting the resulting expression for $(\bar{Q}_x^t)^\top$ into \eqref{eq:lambdataylor} yields
    \begin{equation}
        \setlength\arraycolsep{0pt}
        \begin{array}{ccl}
        \lambda^t(\mathbf{w}^\star) & = & (\bar{V}_x^t)^\top + (C_t + \beta_t^\top B_t)(x_t^\star - \bar{x}_t) - \omega_t^\top \bar{c}^t \\
        & & \quad + (\beta^\top H_t + B_t^\top)(u_t^\star - \bar{u}_t) \\
        & & \quad + (A_t^\top \beta_t + \bar{c}_x^t)^\top (\phi_t^\star - \bar{\phi}) + R \\
        & \overset{\eqref{eq:bwkktfull}}{=} & (\bar{V}_x^t)^\top -\omega_t^\top A_t^\top (u_t^\star - \bar{u}_t) \\
        & & \quad + (C_t + \beta_t^\top B_t) (x_t^\star - \bar{x}_t) -  \omega_t^\top \bar{c}^t + R\\ 
        & \overset{\eqref{eq:kktpritaylor}}{=} & (\bar{V}_x^t)^\top + (C_t + \beta_t^\top B_t + \omega_t^\top \bar{c}_x^t)(x_t^\star - \bar{x}_t) + R\\
        & \overset{\eqref{eq:bwkktfull}}{=} & (\bar{V}_x^t)^\top + (C_t + \beta_t^\top B_t \\
        & & \quad - \omega_t^\top A_t^\top \beta_t)(x_t^\star - \bar{x}_t) + R\\
        & \overset{\eqref{eq:bwkktfull}, \eqref{eq:vxxupdate}}{=} & (\bar{V}_x^t)^\top + P_t (x_t^\star - \bar{x}_t) + R,
        \end{array}
    \end{equation}
    as required.
\end{proof}

\begin{remark}
    Stacking \eqref{eq:ghattaylor} and \eqref{eq:kktpritaylor} yields the identity
    \begin{equation}\label{eq:kktrhsexpansion}
            \begin{bmatrix}
            (\bar{Q}_u^t)^\top \\
            \bar{c}^t
            \end{bmatrix} = - K_t 
            \begin{bmatrix}
                u_t^\star - \bar{u}_t \\ \phi_t^\star - \bar{\phi}_t
            \end{bmatrix}
            - \begin{bmatrix}
                B_t \\ 
                \bar{c}_x^t
            \end{bmatrix} (x_t^\star - \bar{x}_t) + R.
    \end{equation}
\end{remark}
\medskip

Finally, we present our main result, which shows that when the current iterate is sufficiently close to an optimal point $\mathbf{w}^\star$, FilterDDP converges at a quadratic rate. Let $\mathbf{w}^+$ be shorthand for $\mathbf{w}^+(1)$, i.e., a full forward pass step where $\gamma=1$. 
\begin{theorem}
    Under Assumptions 1, there exists positive scalars $\epsilon$ and $M$ such that if $\|\mathbf{w}^\star - \bar{\mathbf{w}}\| < \epsilon$, 
    \begin{equation}\label{eq:quadraticrate}
        \|\mathbf{w}^\star - \mathbf{w}^+\| \leq M \|\mathbf{w}^\star - \bar{\mathbf{w}} \|^2,
    \end{equation}
    and furthermore,
    \begin{equation}\label{eq:contract}
        \|\mathbf{w}^\star - \mathbf{w}^+\| < \|\mathbf{w}^\star - \bar{\mathbf{w}} \|.
    \end{equation}
\end{theorem}
\begin{proof}
    Assumptions \ref{asm:pdhessian} and \ref{asm:Atfullrank} ensures $K_t$ in \eqref{eq:bwkktfull} is non-singular for all $t\in[N]$. The forward pass yields
    \begin{equation}\label{eq:forwardlinear}
    \setlength\arraycolsep{0pt}
    \begin{array}{ccl}
        \begin{bmatrix}
            u_t^+ - \bar{u}_t \\
            \phi_t^+ - \bar{\phi}_t
        \end{bmatrix}
        & \overset{\eqref{eq:forwardpass}}{=} & \begin{bmatrix}
            \alpha_t \\
            \psi_t
        \end{bmatrix} + 
        \begin{bmatrix}
            \beta_t \\ \omega_t
        \end{bmatrix} (x_t^+ - \bar{x}_t) \\
        & \overset{\eqref{eq:bwkktfull}}{=} & - (K_t)^{-1} \begin{bmatrix}
            (\bar{Q}_u^t)^\top \\
            \bar{c}^t
        \end{bmatrix} - (K_t)^{-1}\begin{bmatrix}
            B_t \\
            \bar{c}_x^t 
        \end{bmatrix} (x_t^+ - \bar{x}_t) \\
        & \overset{\eqref{eq:kktrhsexpansion}}{=} & \begin{bmatrix}
            u_t^\star - \bar{u}_t \\ \phi_t^\star - \bar{\phi}_t
        \end{bmatrix} - (K_t)^{-1}  \begin{bmatrix}
            B_t \\ 
            \bar{c}_x^t
        \end{bmatrix} (x_t^+ - x_t^\star) + R,
    \end{array}
    \end{equation}
    noting that the $\bar{x}_t$ terms cancel in simplifying line 3 of \eqref{eq:forwardlinear}.
    
    Since $(K_t)^{-1}$, $B_t$ and $\bar{c}_x^t$ are uniformly bounded by Assumptions 1, it follows after rearranging \eqref{eq:forwardlinear} that
    \begin{equation}\label{eq:uphiplus}
        \begin{gathered}
            u_t^+ - u_t^\star = O\left(\max (\|x_t^+ - x_t^\star\|, \|\mathbf{w}^\star - \bar{\mathbf{w}}\|^2)\right) \\
            \phi_t^+ - \phi_t^\star = O\left(\max (\|x_t^+ - x_t^\star\|, \|\mathbf{w}^\star - \bar{\mathbf{w}}\|^2)\right).
        \end{gathered}
    \end{equation}

    It remains to show that $x_t^+ - x_t^\star = R $ for all $t\in[N]$, which we do using an inductive argument. The base case of $t=1$ holds since $x_1^\star = x_1^+ = \hat{x}_1$. Now suppose $x_t^+ - x_t^\star = R$ for some $t\in[N-1]$. Then by \eqref{eq:uphiplus}, it follows that $w_t^+ - w_t^\star = R$, thus $w_t^+ - \bar{w}_t = w_t^\star - \bar{w}_t + R$ and so $w_t^+ - \bar{w}_t = O(\|\mathbf{w}^\star - \bar{\mathbf{w}}\|)$. 
    
    In addition to \eqref{eq:dyntaylor}, Taylor's theorem applied to $f$ yields
    \begin{equation}\label{eq:dyntaylorplus}
        x_{t+1}^+ = \bar{x}_{t+1} + \bar{f}_x^t (x_t^+ - \bar{x}_t) + \bar{f}_u^t(u_t^+ - \bar{u}_t) + R,
    \end{equation}
     from which we subtract \eqref{eq:dyntaylor} and use $w_t^+ - w_t^\star = R$ to yield $x_{t+1}^+ - x_{t+1}^\star = R$. Additionally, it follows from \eqref{eq:uphiplus} that $u_{t+1}^+ - u_{t+1}^\star = R$ and $\phi_{t+1}^+ - \phi_{t+1}^\star = R$, and so $w_{t+1}^+ - w_{t+1}^\star = R$. We conclude that $\mathbf{w}^+ - \mathbf{w}^\star = R$, and by setting $M$ as the constant in $R$, we have shown that \eqref{eq:quadraticrate} holds.

     Finally, to show \eqref{eq:contract}, we select $\epsilon < 1 / M$ and substitute $\|\mathbf{w}^\star - \bar{\mathbf{w}}\| < \epsilon$ into \eqref{eq:quadraticrate} for the desired result.
\end{proof}

\begin{remark}
    Replacing $\lambda^t$ with $V^t$ invalidates Prop. \ref{prop:quadtaylor}, since \eqref{eq:lambdaw} evaluated at $\bar{\mathbf{w}}$ is in general, inconsistent with \eqref{eq:vxupdate}.
\end{remark}

\section{Extension for Inequality Constraints}

In this section, we present an extension of Alg. 1 for solving inequality constrained problems of form
\begin{equation}\label{eq:oc_ineq}
\begin{array}{rl}
    \underset{\mathbf{x}, \mathbf{u}}{\text{minimise}}  & J(x_{1:N}, u_{1:N}) \coloneqq \sum_{t=1}^{N} \ell(x_t, u_t) \\
    \text{subject to} & x_1 = \hat{x}_1, \\
    & x_{t+1} = f(x_t, u_t) \quad \, \text{for } \,\,\,\, t \in [N-1], \\
    & c(x_t, u_t) = 0, \quad u \geq 0,  \quad \text{for } t \in [N].
\end{array}
\end{equation}
A simple way to extend Alg. 1 is to use a (primal) barrier interior point method as in \cite{prabhu2024differentialdynamicprogrammingstagewise}, which approximately solves a sequence of equality constrained barrier sub-problems, indexed by barrier parameters $\{\mu_j\}$, with the property that $\lim_{j\rightarrow \infty} \mu_j = 0$. Each sub-problem is given by \eqref{eq:oc}, after replacing running cost $\ell$ with the \emph{barrier cost}
\begin{equation}
    \varphi_{\mu_j}(x_t, u_t) \coloneqq \ell(x_t, u_t) - \mu \sum_{i=1}^{m} \ln(u_t^{(i)}).
\end{equation}

Instead, our proposed extension adopts a primal-dual interior point framework, similar to \cite{pavlov_ipddp} and \cite{ipopt}. Concretely, we introduce dual variables for the inequality constraints denoted by $\mathbf{z}\coloneqq z_{1:N}$. The barrier sub-problems described above are then replaced with the \textit{perturbed KKT conditions} given by \eqref{eq:kkt}, $ z_t^\star \odot u_t^\star - \mu_j e = 0$ and $z_t^\star \geq 0$. Primal-dual methods are in general, numerically better conditioned than their primal counterparts e.g., see \cite[Ch. 19]{nwopt} and \cite{ipopt, pavlov_ipddp}.

\subsection{Changes to the Backward and Forward Passes}\label{ssec:changesipm}

The primal-dual extension of FilterDDP requires minor changes to various equations in Sec. \ref{sec:filterddp}, which are described in this section. First, the termination criteria \eqref{eq:opterr} for sub-problem $j$ is replaced by $E_{\mu_j}(\mathbf{w}, \bm{\lambda}) < \kappa_{\epsilon} \mu_j$, where
\begin{equation}\label{eq:terminationipm}
    E_{\mu_j}(\mathbf{w}, \bm{\lambda}) \coloneqq \max\left\{ E(\mathbf{w}, \bm{\lambda}), \max_t \|z_t \odot u_t - \mu_j e\| \right\},
\end{equation}
and $\kappa_\epsilon > 0$ is a parameter. Simply put, the perturbed complementarity slackness conditions are added.

If \eqref{eq:terminationipm} is satisfied for the current iterate, then the sub-problem is considered solved and $\mu_j$ is updated using
\begin{equation}\label{eq:fiaccomccormick}
    \mu_{j+1} = \max \left\{ \frac{\epsilon_\text{tol}}{10}, \min \left\{ \kappa_\mu \mu_j, \mu_j^{\theta_\mu}\right\}\right\},
\end{equation}
where $\kappa_\mu \in (0, 1)$ and $\theta_\mu \in (1, 2)$. This formula was proposed by~\cite{byrdsuper} and used in IPOPT~\cite{ipopt}, and encourages a superlinear decrease of $\mu_j$. The algorithm is deemed to have converged if $E_0(\mathbf{w}_k, \bm{\lambda}_k) < \epsilon_\text{tol}$.

\paragraph{Backward Pass} We replace \eqref{eq:bwkktfull} with
\begin{equation}\label{eq:bwipm}
    \begin{bmatrix}
        H_t + \Sigma_t + \delta_w I & A_t \\
        A_t^\top & -\delta_c I
    \end{bmatrix}
    \begin{bmatrix}
        \alpha_t & \beta_t \\ \psi_t & \omega_t
    \end{bmatrix}
     = -\begin{bmatrix}
         (\hat{Q}_u^t)^\top & B_t \\ \bar{c}^t & \bar{c}_x^t
     \end{bmatrix},
\end{equation}
where $\bar{U}_t \coloneqq \mathrm{diag}(\bar{u}_t)$, $\bar{Z}_t \coloneqq \mathrm{diag}(\bar{z}_t)$,  $\Sigma_t \coloneqq \bar{U}_t^{-1} \bar{Z}_t$ and $\hat{Q}_u^t \coloneqq \bar{Q}_u^t - \mu e^\top \bar{U}_t^{-1}$.

\paragraph{Forward Pass} In addition to \eqref{eq:forwardpass}, we apply the update rule for dual variables $z_t$ given by
\begin{equation}
    z_t^+(x_t^+, \gamma) = \bar{z}_t + \gamma\chi_t + \zeta_t (x^+_t - \bar{x}_t),
\end{equation}
where $\chi_t = \mu \bar{U}_t^{-1} e - \bar{z}_t - \Sigma_t \alpha_t$ and $\zeta_t = -\Sigma_t \beta_t$.
Also, a fraction-to-the-boundary condition \cite{ipopt} given by
\begin{equation}\label{eq:fractoboundary}
    u_t^+(x_t^+, \gamma) \geq (1 - \tau)\bar{u}_t, \quad z_t^+(x_t^+, \gamma) \geq (1 - \tau) \bar{z}_t \quad \forall t,
\end{equation}
where $\tau = \max\{\tau_\mathrm{min}, 1-\mu\}$ for some $\tau_\mathrm{min} > 0$ must be satisfied for a step size $\gamma$ to be accepted as the next iterate.

Finally, for filter condition \eqref{eq:filtercriteria}, we replace $\ell$ with $\varphi_{\mu_j}$ for sub-problem $j$ and replace $\bar{Q}_u^t$ with $\hat{Q}_u^t$ in \eqref{eq:suffdecreasecost}.

\paragraph{Barrier sub-problem update step} The final modification to Alg. 1 is that Step \ref{alg:checkconverge} of Alg. 1 is replaced with:
\begin{enumerate}[label*=\arabic*.]
    \setcounter{enumi}{2}
    \item \textit{Check convergence of overall problem.} \,\, If $E_0(\mathbf{w}_{k}$, $\bm{\lambda}_k) < \epsilon_\mathrm{tol}$ then STOP (converged). If $k = \text{max\_iters}$ then STOP (not converged). \label{alg:checkconvergeipm}
    \begin{enumerate}[label*=\arabic*.]
        \item \textit{Check convergence of barrier problem.} \,\, If $E_{\mu_j}(\mathbf{w}_{k}$, $\bm{\lambda}_k) < \kappa_\epsilon \mu_j$ then:
        \begin{enumerate}[label*=\arabic*.]\label{alg:updatebarrier}
            \item Compute $\mu_{j+1}$ and $\tau_{j+1}$ from \eqref{eq:fiaccomccormick} and \eqref{eq:fractoboundary} and set $j\gets j+1$.
            \item Re-initialise the filter  $\mathcal{F}_k \coloneqq \left\{ (\theta, \mathcal{L}) \in \mathbb{R}^2 : \theta \geq \theta_{\max} \right\}$.
            \item If $k=0$, repeat step \ref{alg:updatebarrier}, otherwise continue at \ref{alg:backtrackingls}
        \end{enumerate}
    \end{enumerate}
    \label{alg:checkconvergebarrier}
\end{enumerate}
Finally, Alg. 1 requires additional parameters $\tau_\mathrm{min} > 0$, $\mu_\mathrm{init} > 0$, $\kappa_\epsilon > 0$, $\kappa_\mu \in (0, 1)$ and $\theta_\mu \in (1, 2)$.

\begin{figure*}[t!]
     \centering
     \begin{subfigure}[b]{0.3\textwidth}
         \centering
         \includegraphics[width=\textwidth]{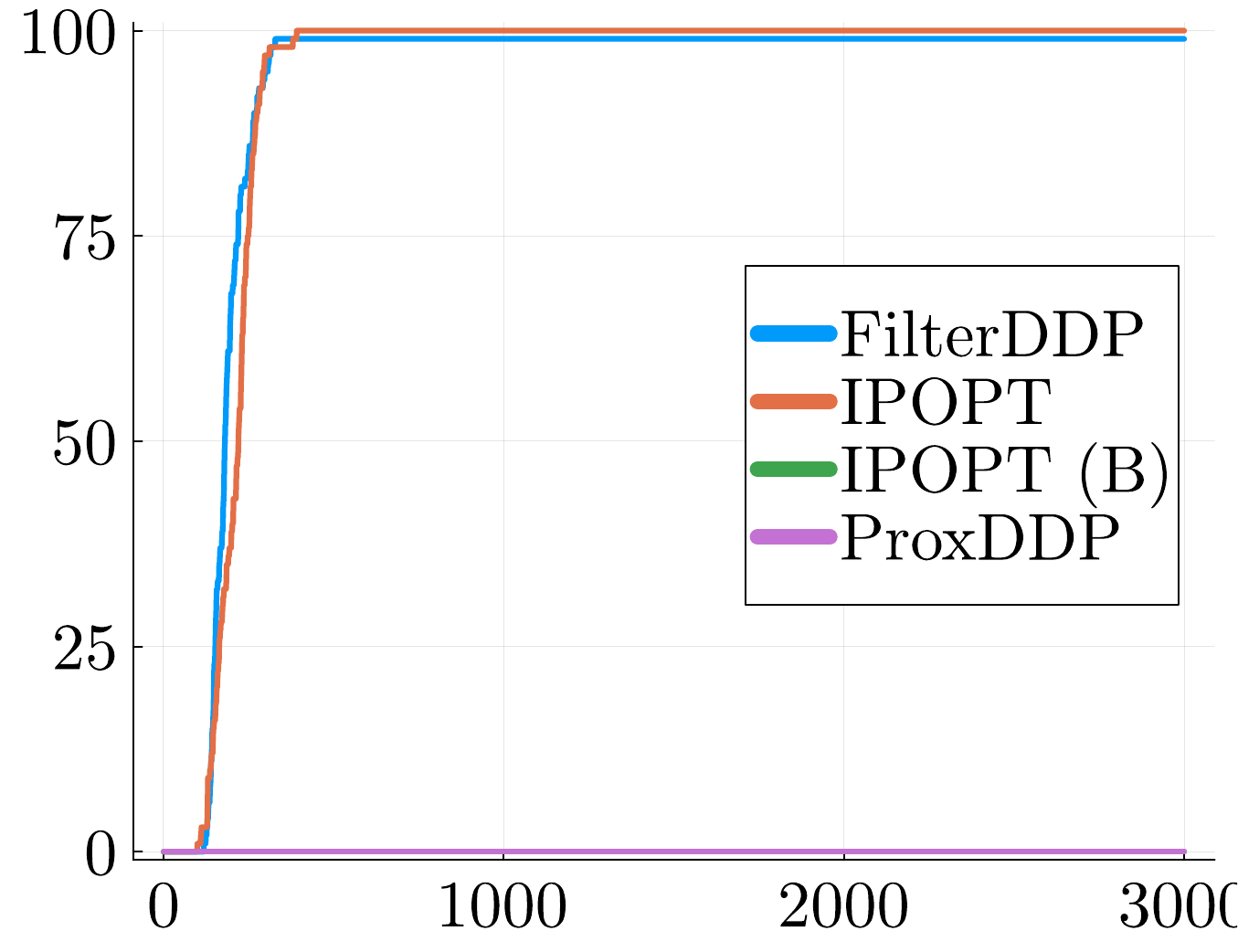}
         \caption{Acrobot Contact}
         \label{fig:results_cartpole}
     \end{subfigure}
     \hfill
     \begin{subfigure}[b]{0.3\textwidth}
         \centering
         \includegraphics[width=\textwidth]{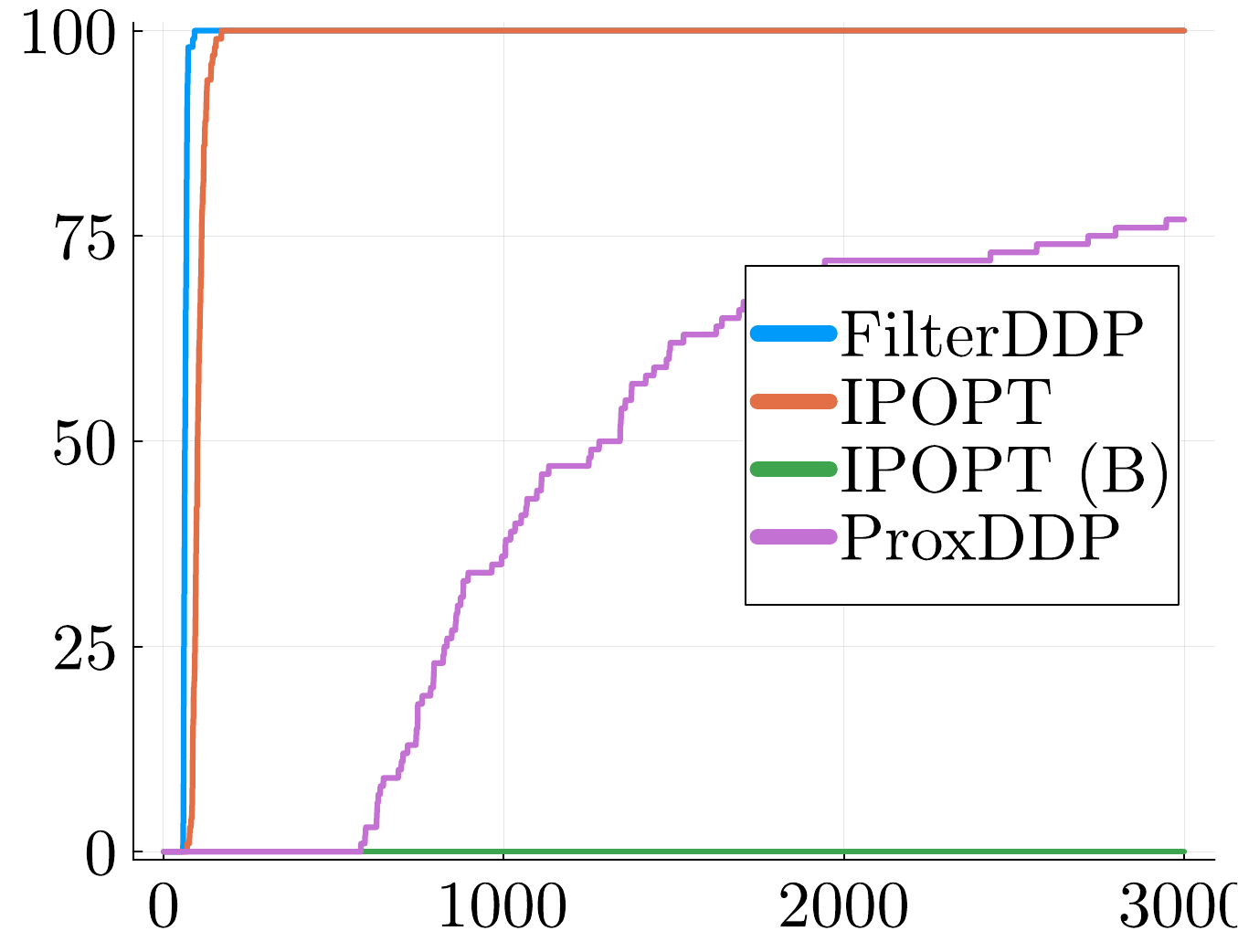}
         \caption{Cartpole Friction}
         \label{fig:results_acrobot}
     \end{subfigure}
     \hfill
     \begin{subfigure}[b]{0.3\textwidth}
         \centering
         \includegraphics[width=\textwidth]{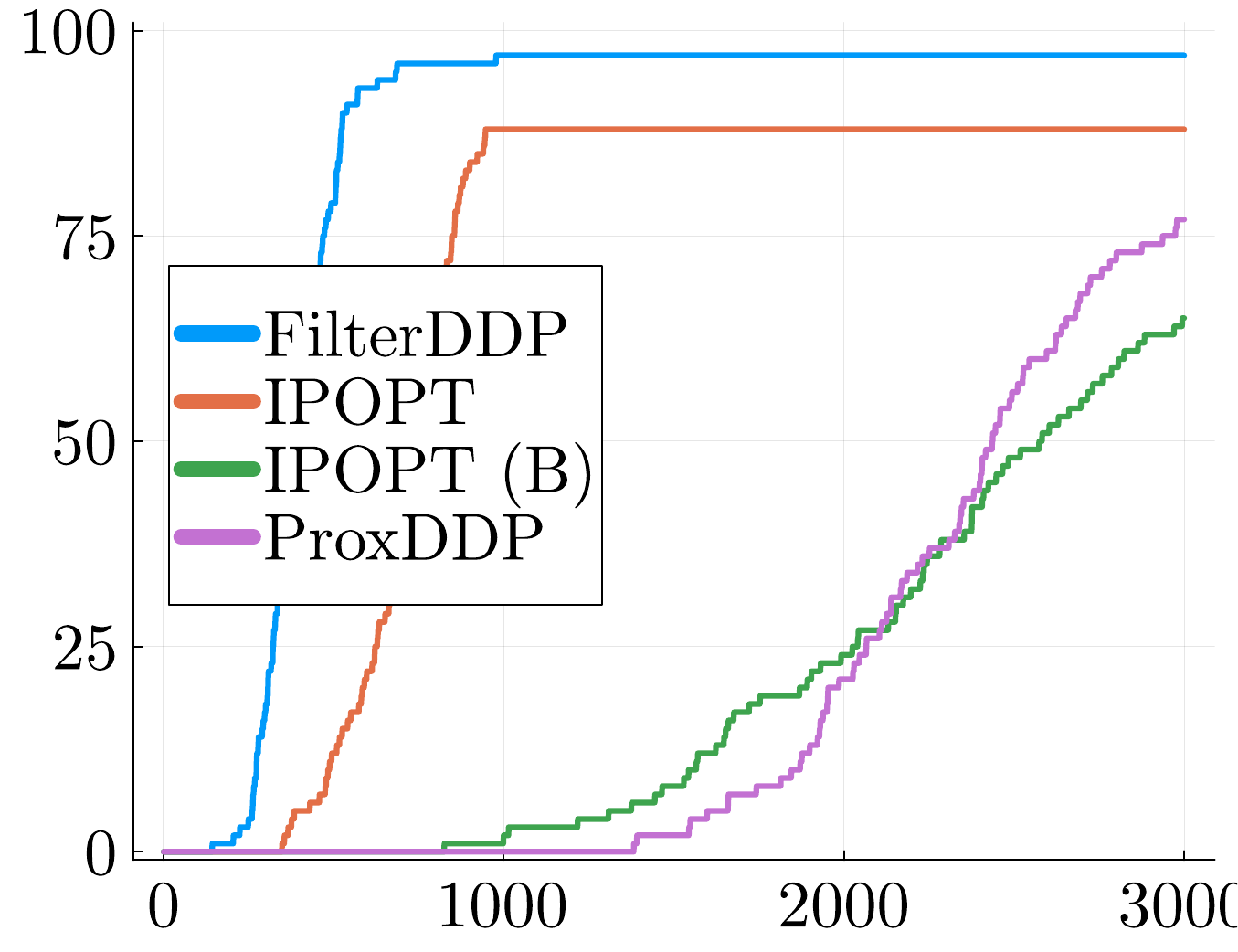}
         \caption{Manipulation}
         \label{fig:results_mani}
     \end{subfigure}
     \caption{Results for the three planning tasks. For a)-c), the x-axis represents iteration count and the y-axis is the number of OCPs which converged to the error tolerance of $10^{-7}$ for Filter DDP and IPOPT and $10^{-5}$ for ProxDDP and IPOPT (B). }
     \label{fig:results}
\end{figure*}

\begin{figure*}[t!]
     \centering
     \begin{subfigure}[b]{0.3\textwidth}
         \centering
         \includegraphics[width=\textwidth]{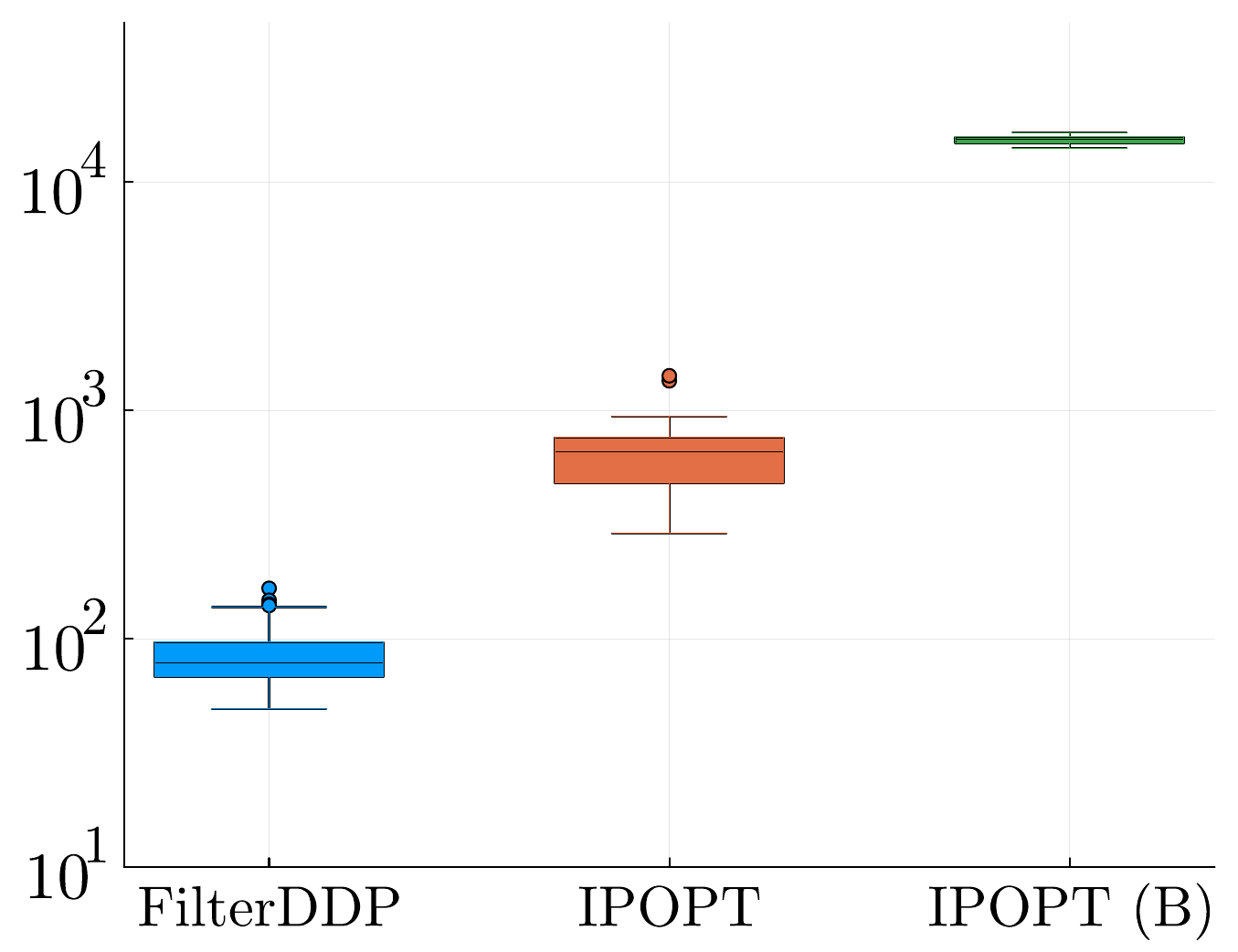}
         \caption{Acrobot Contact}
         \label{fig:time_cartpole}
     \end{subfigure}
     \hfill
     \begin{subfigure}[b]{0.3\textwidth}
         \centering
         \includegraphics[width=\textwidth]{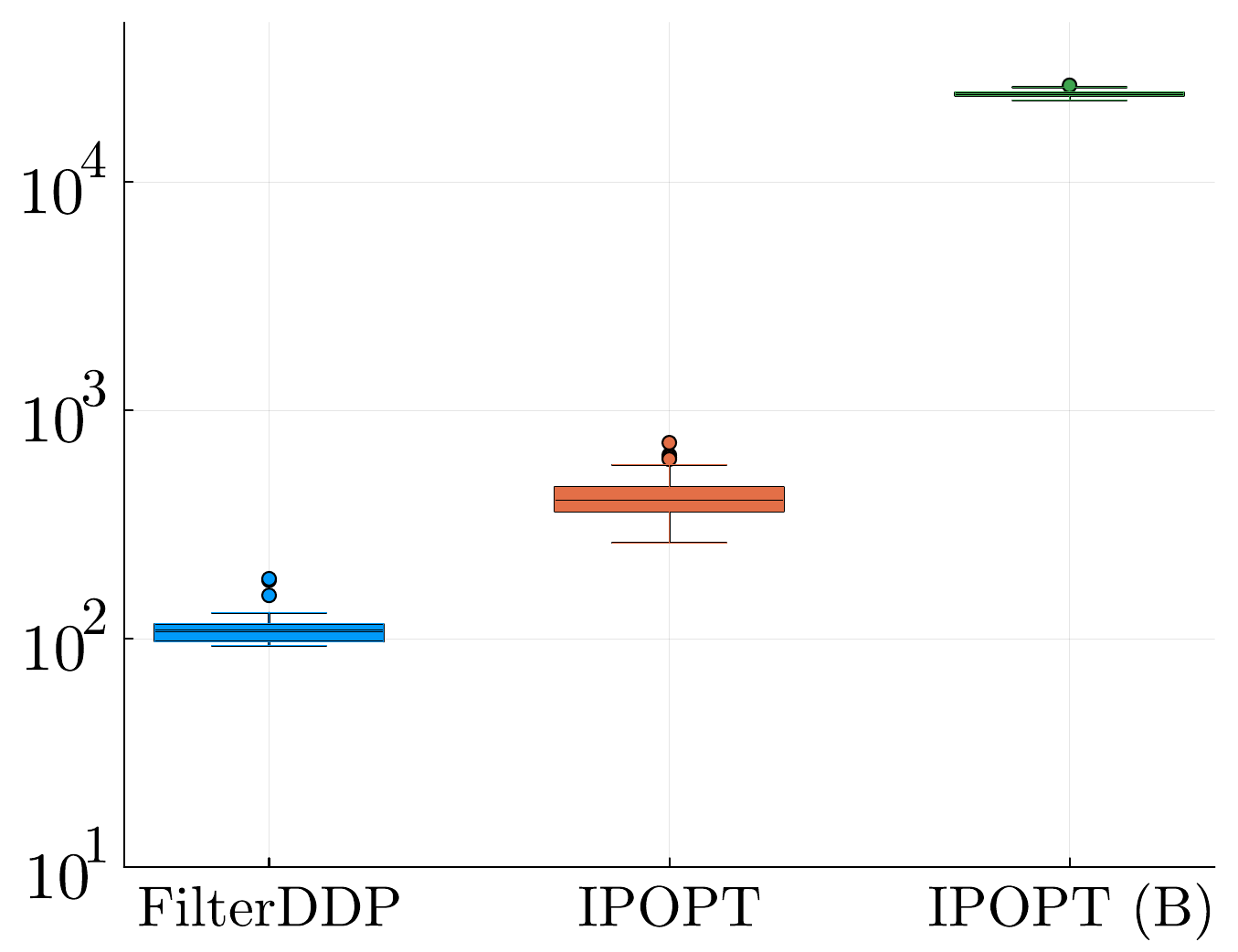}
         \caption{Cartpole Friction}
         \label{fig:time_acrobot}
     \end{subfigure}
     \hfill
     \begin{subfigure}[b]{0.3\textwidth}
         \centering
         \includegraphics[width=\textwidth]{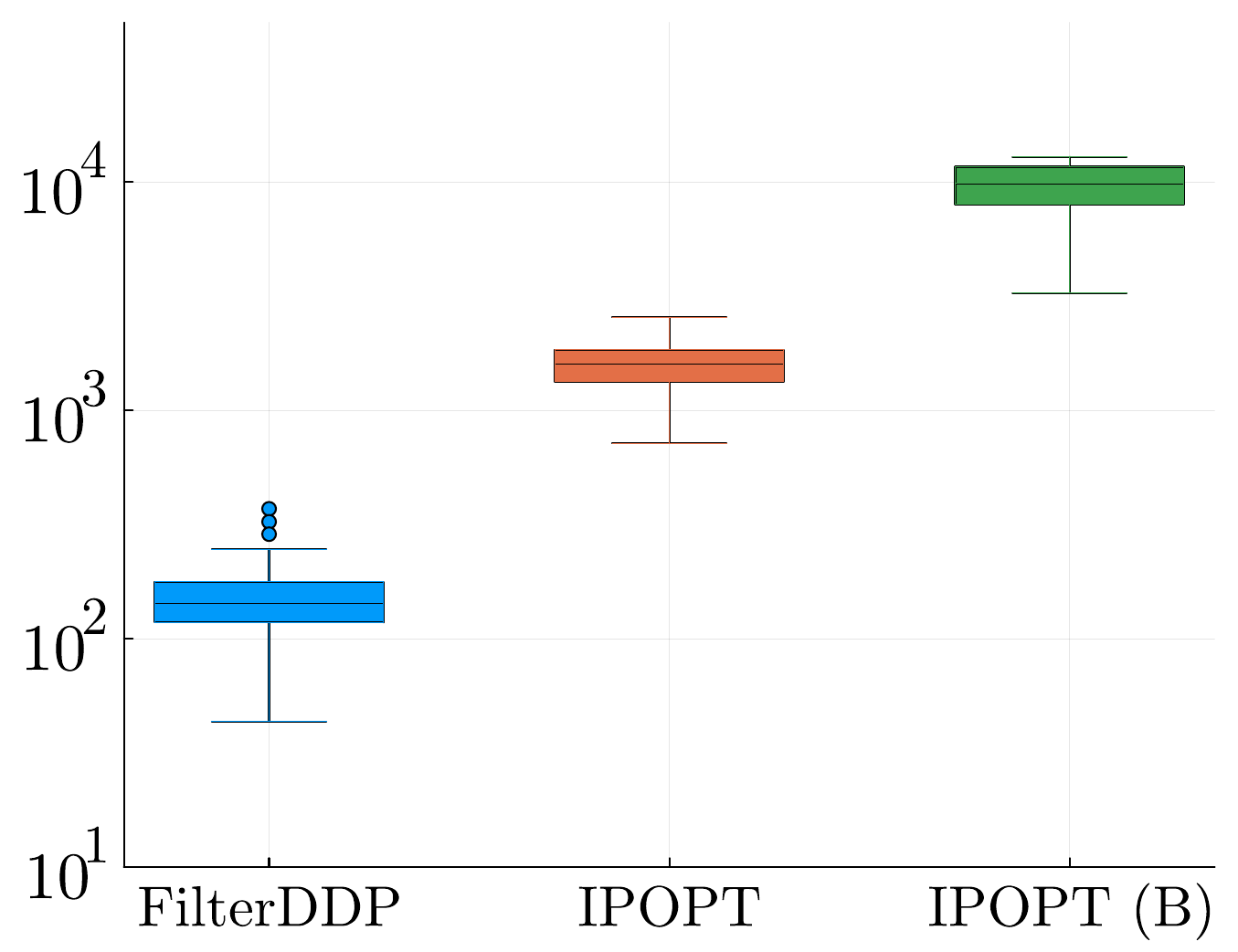}
         \caption{Manipulation}
         \label{fig:time_mani}
     \end{subfigure}
     \caption{Total wall clock (ms) for each method over each OCP class. FilterDDP attains a mean wall clock time of 14.5\% and 0.6\% of IPOPT and IPOPT (B), respectively for the acrobot problem, a mean wall clock time of 27.1\% and 4.5\% of IPOPT and IPOPT (B), respectively for the cartpole problem and finally a mean wall clock time of 10.3\% and 1.7\% of IPOPT and IPOPT (B), respectively for the manipulation problem.}
     \label{fig:results_timings}
\end{figure*}

\section{Numerical Simulations}\label{sec:experiments}

We evaluate FilterDDP on three planning tasks with nonlinear equality constraints: 1) a contact-implicit cartpole swing-up task with Columb friction, 2) a contact-implicit acrobot swing-up task with hard joint limits and, 3) a contact-implicit non-prehensile manipulation task.

\subsection{Comparison Methods}

FilterDDP is compared against IPOPT~\cite{ipopt} with full second-order derivatives and an L-BFGS variant which we refer to as IPOPT (B), as well the AL-DDP algorithm ProxDDP~\cite{proxddp}. FilterDDP uses the Julia Symbolics.jl package \cite{gowda2021high}, ProxDDP uses CasADi~\cite{casadi} and IPOPT uses the JuMP.jl package \cite{Lubin2023} for computing function derivatives. For the acrobot and manipulation tasks, we implement FilterDDP using statically-sized arrays and manual implementations of linear algebra operations\footnote{https://github.com/JuliaArrays/StaticArrays.jl} (avoiding external libraries such as LAPACK \cite{lapack}), which is more efficient for small numbers of states, control variables and constraints thanks to compiler optimisations. For cartpole however, we instead implement FilterDDP using LAPACK for numerical linear algebra, which is more efficient for larger problem sizes. In our experiments, we report total wall clock time for FilterDDP and IPOPT. For ProxDDP, we do not report timing results since we used the (slower) Python interface. 

\subsection{Experimental Setup}

The solver parameters selected for FilterDDP across all experiments are identical to their IPOPT counterparts described in \cite{ipopt}, except we reduce the superlinear barrier update term to $\theta_\mu = 1.2$ and increase the default barrier parameter to $\mu_\mathrm{init} = 1$. FilterDDP are limited to 1,000 iterations while ProxDDP and IPOPT (B) are limited to 3,000 iterations. Optimality error tolerances are set to $10^{-7}$ for IPOPT and FilterDDP and $10^{-5}$ for ProxDDP and IPOPT (B). For ProxDDP, we applied a grid search per planning task to select the initial AL parameter $\mu_0$ to minimise constraint violation, while default parameters were used for FilterDDP, IPOPT and IPOPT (B). All specific parameter values are provided in the Julia code. For each task, we create 100 individual OCPs by varying the dynamics parameters (e.g., link lengths and masses for acrobot) and actuation limits. All experiments were limited to one CPU core and performed on a Lenovo ThinkPad T14s Gen 5 with an Intel® Core™ Ultra 7 155U and 32GB of RAM, running Ubuntu 24.04.3 LTS and Julia version 1.13.0-beta3 compiled using the -O3 optimisation level.

\subsection{Cartpole Swing-Up with Coulomb Friction}

The first task is the cartpole swing-up task with Coulomb friction applying to the cart and pendulum, introduced by \cite{howellopt}. Frictional forces obey the principle of maximum dissipation, and the resultant OCP is derived by a variational time-stepping framework \cite{manchestercontactto, lecleachcimpc}. Frictional forces are resolved jointly with the rigid-body dynamics by introducing constraints.

The states $x\in\mathbb{R}^4$ include the prior and current configurations of the system, while $u\in\mathbb{R}^{15}$ includes the cart actuation force, the next configuration of the system, frictional forces and finally, dual and slack variables relating to the contact constraints \cite{manchestercontactto}. For each OCP, we vary the mass of the cart and pole, length of the pole and friction coefficients. See \cite{howellopt} and the supplementary for a detailed description of the OCP.

\subsection{Acrobot Swing-Up with Joint Limits}\label{ssec:acrobot}

This task is an acrobot swing-up task, where limits at $+\pi/2$ and $-\pi/2$ are imposed on the elbow joint, introduced by \cite{howellopt}. The joint limits are enforced by an impulse, which is only applied when a joint limit is reached. Intuitively, the acrobot is allowed to ``slam" into a joint limit, and the impulse takes on the appropriate value to ensure the joint limits are not violated. Similar to cartpole, a variational time-stepping framework \cite{manchestercontactto, howellopt} is applied to encode the contact dynamics. 

The states $x\in\mathbb{R}^4$ include the prior and current configurations of the system, while $u\in\mathbb{R}^{7}$ includes the elbow torque, the next configuration, impulses at the joint limits and slack variables for the complementarity constraints \cite{manchestercontactto}. For each OCP, we vary the link lengths. See \cite{howellopt} and the supplementary for a description of the OCP. An example trajectory of impulses at the joint limits is given in Fig. \ref{fig:qual_acrobot}.

\begin{figure}[h!]
     \centering
      \includegraphics[width=0.32\textwidth]{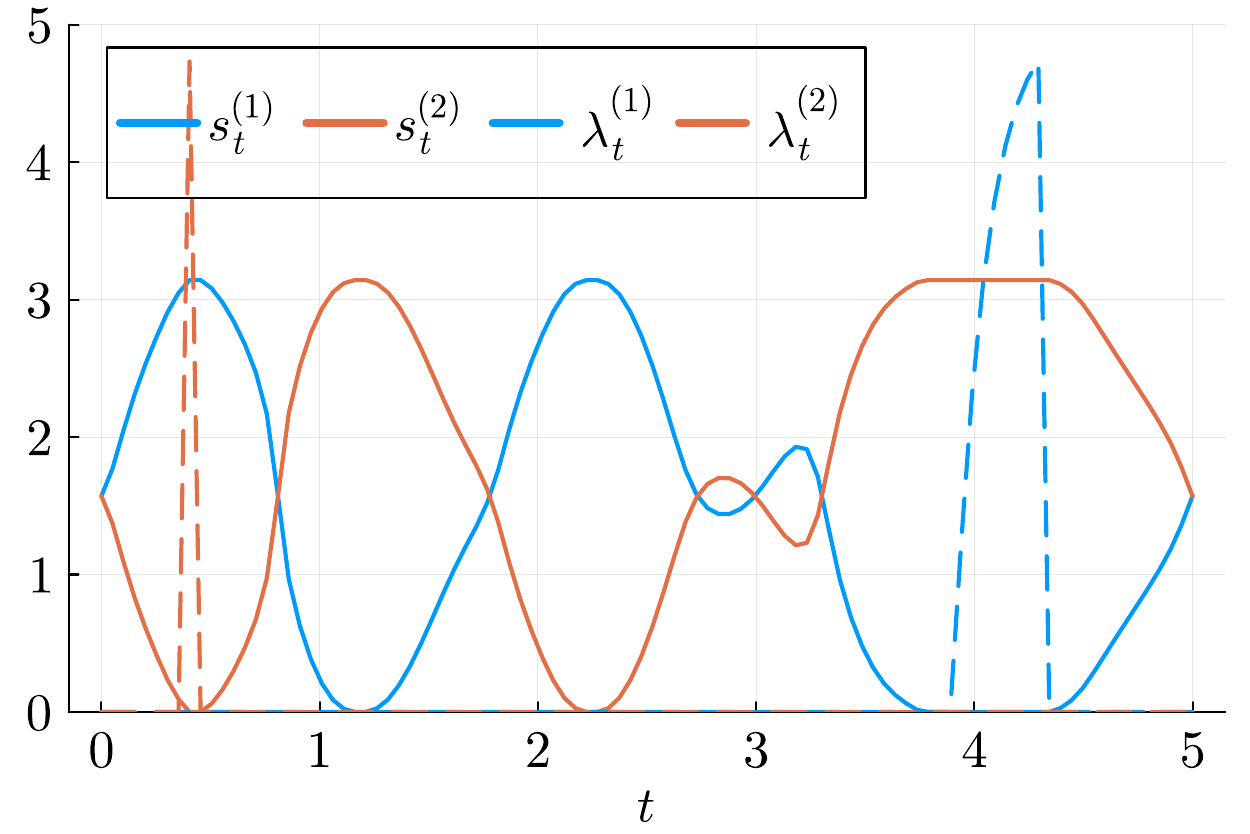}
     \caption{Acrobot example. $s_t$ is the signed distance to the $\pi/2$ and $-\pi/2$ joint limits and $\lambda_t$ denotes the impulses. Around the 4s mark, the acrobot is ``leaning into" the joint limit.}
     \label{fig:qual_acrobot}
\end{figure}

\subsection{Non-Prehensile Manipulation}

\begin{figure}[t!]
     \centering
     \includegraphics[width=0.21\textwidth]{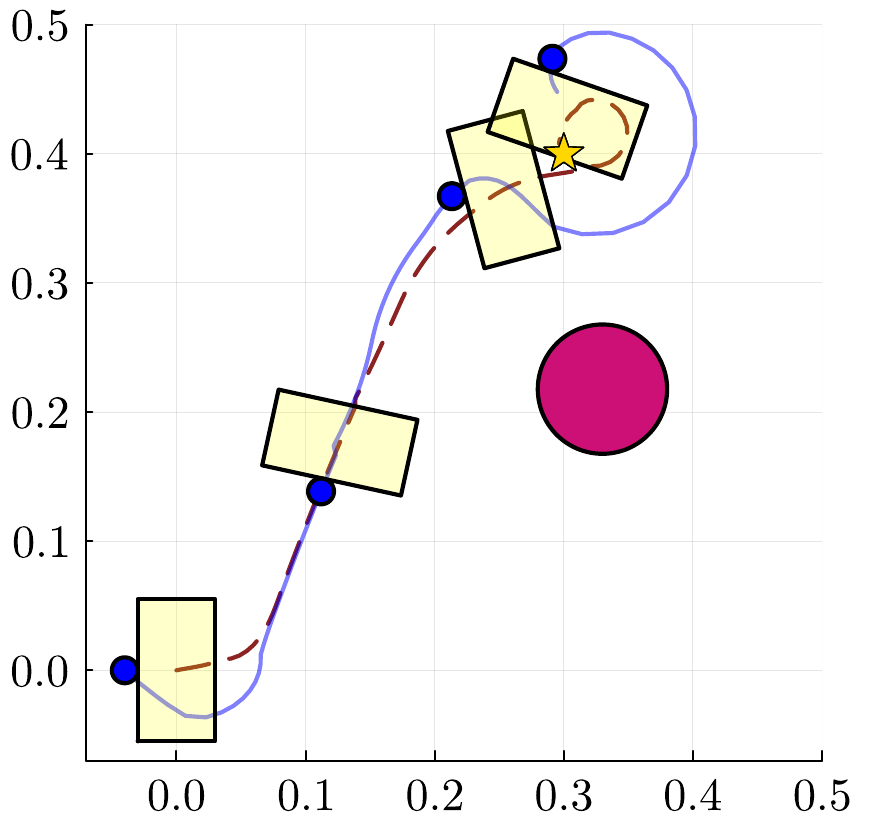}
     \includegraphics[width=0.21\textwidth]{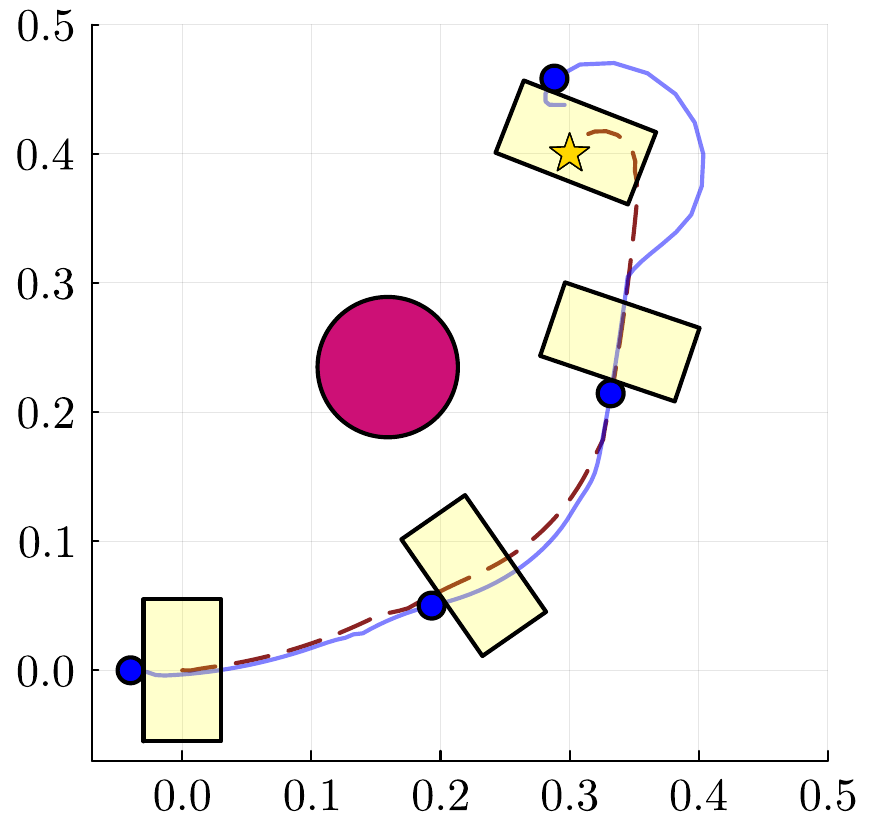}
     \caption{Pusher and slider trajectories for two instances of the non-prehensile manipulation problem. The pusher and  slider trajectories are marked by the blue and brown lines.}
     \label{fig:qual_pushing}
\end{figure}

Our final task is a non-prehensile manipulation planning task, described in detail in \cite{mouranonprehensile}, where a manipulator (the pusher) is tasked with using its end-effector to push a rectangular block (the slider) from a start to a goal configuration, while avoiding an obstacle. The dynamics of the pusher-slider problem includes multiple sticking and sliding contact modes, which are encoded through complementarity constraints. 

The states $x\in\mathbb{R}^4$ include the pose of the slider and pusher in a global coordinate frame. Controls $u\in\mathbb{R}^6$ includes contact forces between the pusher and the slider, relative pusher/slider velocity and slack variables for the obstacle avoidance and complementarity constraints. We also apply inequality constraints to bound the contact force magnitudes and pusher/slider velocity and include an obstacle avoidance constraint. The objective function is a sum of a quadratic penalty on contact force and a high weight linear penalty on the slack variables (a so-called \emph{exact} penalty \cite{manchestercontactto}). For each OCP, we vary the coefficient of friction, obstacle size and location. Fig. \ref{fig:qual_pushing} illustrates sample trajectories.

\begin{figure}[t!]
     \centering
     \begin{subfigure}[b]{0.15\textwidth}
         \centering
         \includegraphics[width=\textwidth]{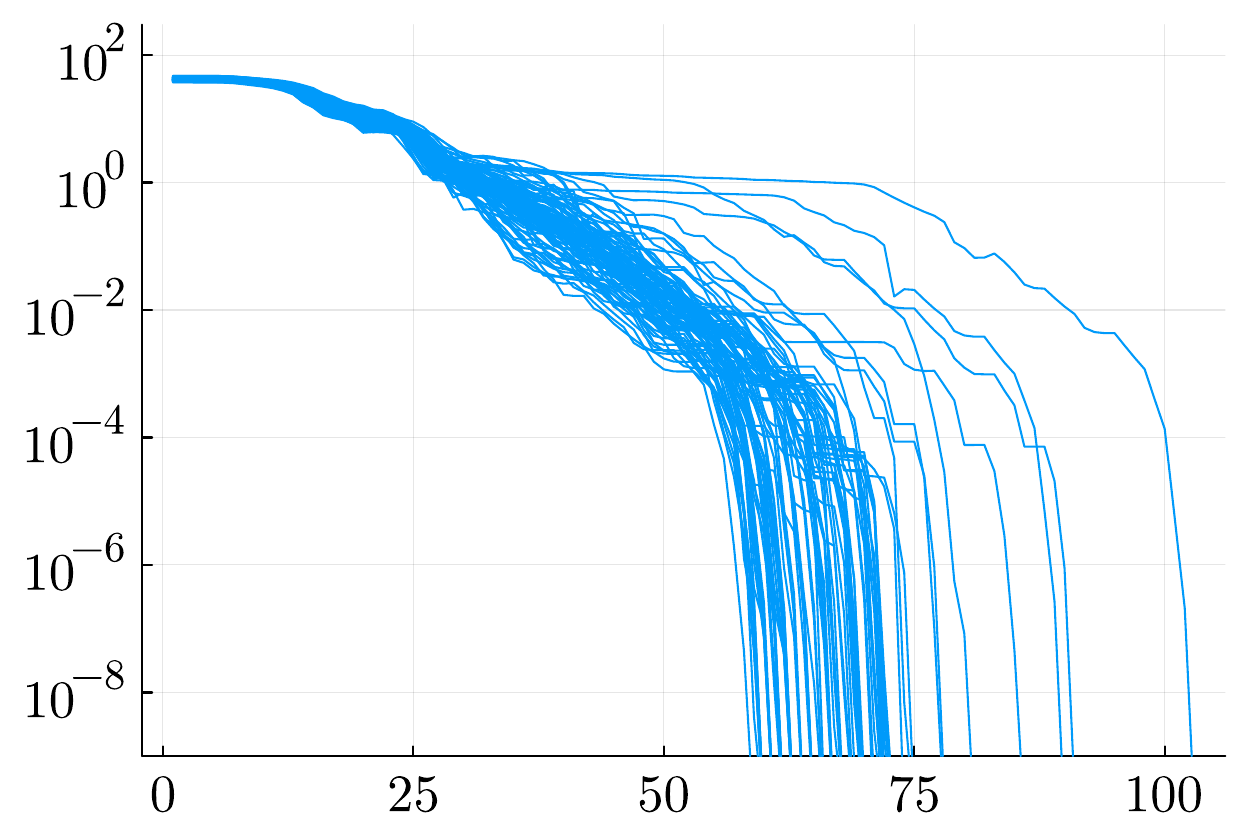}
         \caption{Cartpole Friction}
         \label{fig:cartpole_convergence}
     \end{subfigure}
     \begin{subfigure}[b]{0.15\textwidth}
         \centering
         \includegraphics[width=\textwidth]{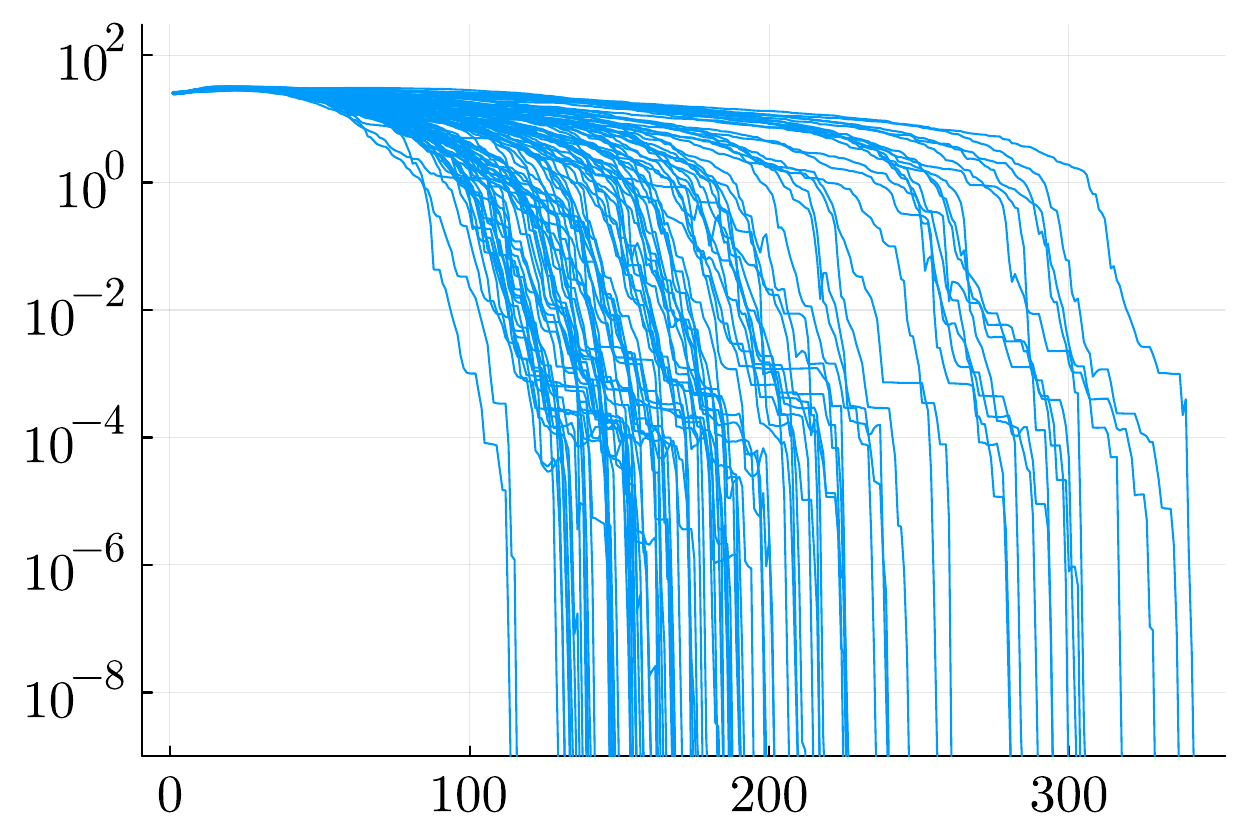}
         \caption{Acrobot Contact}
         \label{fig:acrobot_convergence}
     \end{subfigure}
     \begin{subfigure}[b]{0.15\textwidth}
         \centering
         \includegraphics[width=\textwidth]{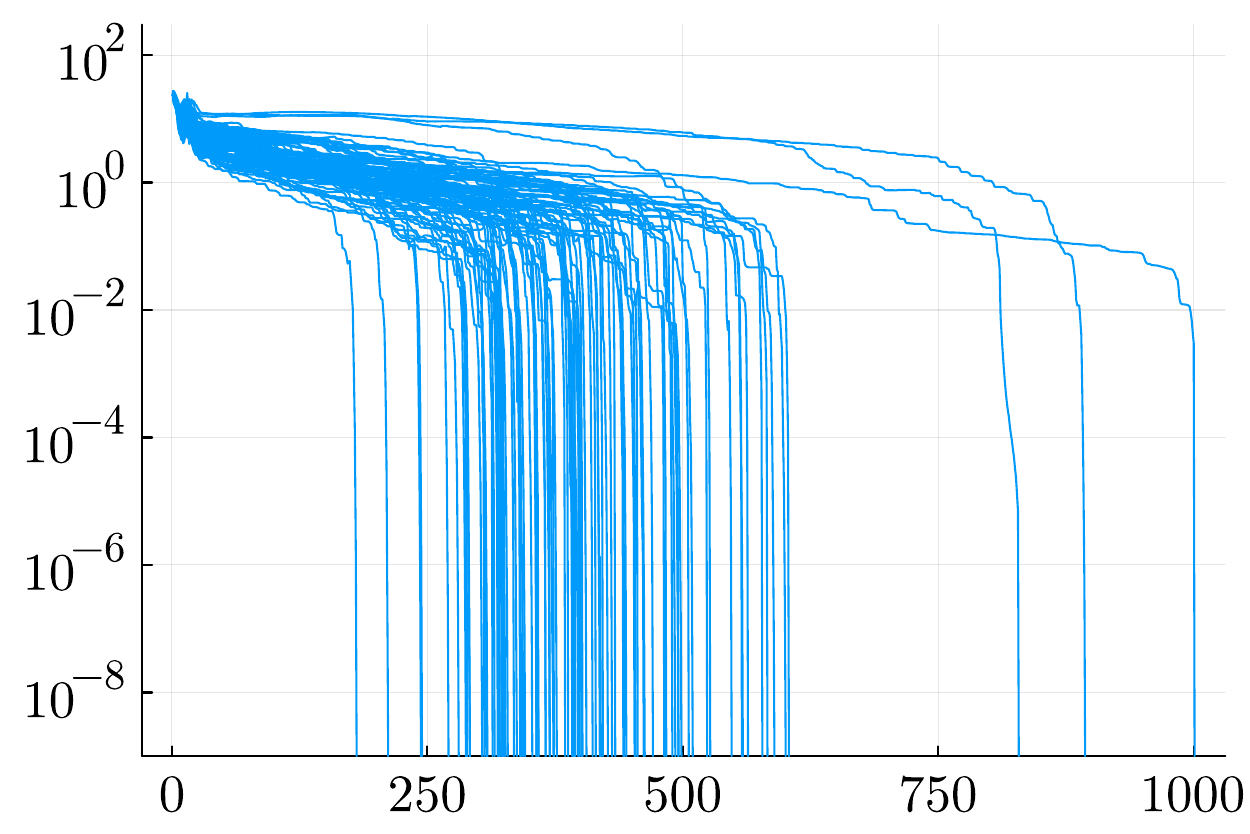}
         \caption{Manipulation}
         \label{fig:pushing_convergence}
     \end{subfigure}
     \caption{Local quadratic convergence of FilterDDP. The x-axis measures iteration count and the y-axis measures $\|\bar{u}_{1:N} - u^\star_{1:N}\|_2$, where $u^\star_{1:N}$ is the optimal point found by FilterDDP.}
     \label{fig:results_convergence}
\end{figure}

\subsection{Summary Discussion of Results}

We present the success rate per planning task across all 100 OCPs as a function of iteration count in Fig. \ref{fig:results}. In addition, we present timing results in Fig. \ref{fig:results_timings}. FilterDDP is significantly faster compared to both IPOPT and IPOPT (B), yielding average wall clock times of 10-27\% of IPOPT and 1-5\% of IPOPT (B) across the three planning tasks. Furthermore, FilterDDP attains a comparable level of robustness to IPOPT in either similar or much fewer iterations. 

Notably, the evaluated first-order methods ProxDDP and IPOPT (B) fails to solve the acrobot example entirely, with ProxDDP yielding feasible but degenerate solutions despite extensive parameter tuning, and IPOPT (B) yielding infeasible solutions. Furthermore, the iteration count for the remaining tasks for both ProxDDP and IPOPT (B) is significantly higher even with a relaxed termination tolerance, with neither method reliably converging to a local optima within a reasonable iteration count. This finding highlights the importance of using a robust second-order solver as opposed to a first-order for solving OCPs with challenging constraints.

In Tab. \ref{tab:ablation}, we report the results of ablating the two critical design choices for the FilterDDP algorithm, 1) replacing the cost with the Lagrangian for step acceptance in \eqref{eq:filtercriteria}, \eqref{eq:armijo} and, 2) replacing Hessians \eqref{eq:qfnderivs} with perturbed Hessians \eqref{eq:bwhessperturbed}. Finally, Fig.~\ref{fig:results_convergence} plots the convergence of FilterDDP for all OCPs across all planning tasks. Local quadratic convergence is observed in practice, validating the formal result in Sec.~\ref{sec:proofconvergence}. 
\begin{remark}
    We experimented with a \emph{Gauss-Newton} approximation to $H_t$ (e.g., as in \cite{mastallinullspace}), however, it is significantly less reliable than using full second-order derivatives, analogous to the performance gap between IPOPT and IPOPT (B).
    
\end{remark}

\begin{table}[h!]
\centering
\begin{tabular}{|c|c|c|c|}
\hline
                                & $\mathcal{L} \rightarrow J$  & $\lambda_t \rightarrow \bar{V}_x^t$ & $ \mathcal{L}, \lambda_t \rightarrow J, \bar{V}_x^t$ \\ \hline
Cartpole \color{blue} (I) \color{black}              & $66 \rightarrow 66$ & $66 \rightarrow 76$ &    $66 \rightarrow 77$    \\ \hline
Acrobot \color{blue} (I) \color{black} & $181 \rightarrow 179$ & $181 \rightarrow 209$ &  $181 \rightarrow 265$      \\ \hline
Manip. \color{blue} (F) \color{black}   & $1 \rightarrow 5$ & $1 \rightarrow 78$ & $1 \rightarrow 63$ \\ \hline
\end{tabular}
\caption{Ablation of algorithm design. \color{blue} (I) \color{black} and \color{blue} (F) \color{black} indicates that iteration count and no. of failures, resp. are reported.}
\label{tab:ablation}
\end{table}

\section{Conclusion and Future Work}\label{sec:conclusion}

In this paper, we have presented a line-search filter differential dynamic programming algorithm for equality constrained optimal control problems. Important design choices for the iterates and filter criterion are proposed and validated both analytically and empirically on contact-implicit trajectory planning problems arising in robotics. Furthermore, a rigorous proof of local quadratic convergence is provided, generalising prior results on unconstrained DDP \cite{liaoddplocal}. Future work will investigate applying FilterDDP to high-dimensional locomotion problems in legged robots, by integration of efficient rigid-body dynamics libraries \cite{carpentierpinocchiolibrary} and demonstrating FilterDDP in a control policy on hardware. Furthermore, formally establishing the global convergence of FilterDDP is another promising direction for future work.

\IEEEpeerreviewmaketitle




\bibliographystyle{ieeetr}
\bibliography{references}

\end{document}